\documentclass{amsart}
\usepackage{verbatim}
\usepackage{rotating} 
\usepackage{amssymb}
\usepackage[bookmarks,colorlinks,breaklinks]{hyperref}  
\hypersetup{linkcolor=blue,citecolor=blue,filecolor=blue,urlcolor=blue} 

\newtheorem{theorem}{Theorem}[section]
\newtheorem{thm}{Theorem}[section]

\newtheorem{prop}[theorem]{Proposition}

\newtheorem{lemma}[theorem]{Lemma}

\newtheorem{cor}[theorem]{Corollary}

\theoremstyle{plain}
\numberwithin{equation}{theorem}

\theoremstyle{remark}
\newtheorem{claim}[theorem]{Claim}
\newtheorem{remark}[theorem]{Remark}

\newtheorem{remarks}[theorem]{Remarks}

\newcommand{\C}{{\mathbb C}}
\renewcommand{\l}{\lambda}

\newcommand{\Q}{{\mathbb Q}}

\newcommand{\Z}{{\mathbb Z}}

\newcommand{\cV}{{\mathcal V}}

\newcommand{\fO}{\mathfrak o}

\newcommand{\pic}{Pic}

\newcommand{\Qbar}{\bar{\Q}}

\DeclareMathOperator{\Gal}{Gal}

\DeclareMathOperator{\cD}{\mathcal{D}}

\DeclareMathOperator{\bN}{\mathbb{N}}

\newcommand{\qbar}{\overline{\Q}}

\newcommand{\bfc}{{\mathbf c}}

\newcommand{\bff}{{\mathbf f}}

\newcommand{\bfA}{{\mathbf A}}
\newcommand{\bfB}{{\mathbf B}}

\newcommand{\bP}{{\mathbb P}}

\newcommand{\lra}{\longrightarrow}

\newcommand{\hhat}{{\widehat h}}

\begin{document}

\title{Variation of the canonical height in a family of rational maps}

\author{D.~Ghioca}
\address{
Dragos Ghioca\\
Department of Mathematics\\
University of British Columbia\\
Vancouver, BC V6T 1Z2\\
Canada
}
\email{dghioca@math.ubc.ca}

\author{N.~M.~Mavraki}
\address{
Niki Myrto Mavraki\\
Department of Mathematics\\
University of British Columbia\\
Vancouver, BC V6T 1Z2\\
Canada
}
\email{myrtomav@math.ubc.ca}

\begin{abstract}
Let $d\ge 2$ be an integer, let $\bfc\in \Qbar(t)$ be a rational map, and  let $\bff_t(z):=\frac{z^d+t}{z}$ be a family of rational maps indexed by $t$. For each $t=\l\in\Qbar$, we let $\hhat_{\bff_\l}(\bfc(\l))$ be the canonical height of $\bfc(\l)$ with respect to the rational map $\bff_\l$; also we let $\hhat_\bff(\bfc)$ be the canonical height of $\bfc$ on the generic fiber of the above family of rational maps. We prove that there exists a constant $C$ depending only on $\bfc$ such that  for each $\l\in\Qbar$, $\left|\hhat_{\bff_\l}(\bfc(\l))-\hhat_\bff(\bfc)\cdot h(\l)\right|\le C$. This improves a result of Call and Silverman \cite{Call-Silverman} for this family of rational maps. 
\end{abstract}

\thanks{2010 AMS Subject Classification: Primary 11G50; Secondary 14G17, 11G10.
 The research of the first author was partially supported by an NSERC grant. The second author was partially supported by Onassis Foundation.}
\maketitle


\section{Introduction}
\label{intro}

Let $X$ be a curve defined over $\Qbar$, let $\cV\lra X$ be an algebraic family of varieties $\{\cV_\l\}_{\l\in X}$, let $\Phi:\cV\lra \cV$ be an endomorphism with the property that there exists $d>1$, and there exists a divisor $\cD$ of $\cV$ such that $\Phi^*(\cD)= d\cdot \cD$. Then for all but finitely many $\l\in X$, there is a well-defined canonical height $\hhat_{\cV_\l,\cD_\l,\Phi_\l}$ on the fiber above $\l$. Let $P:X\lra \cV$ be an arbitrary section; then for each $\l\in X(\Qbar)$, we denote by $P_\l$ the corresponding point on $\cV_\l$. Also, $P$ can be viewed as an element of $V(\Qbar(X))$ and thus we denote by $\hhat_{V,D,\Phi}(P)$ the canonical height of $P$ with respect to the action of $\Phi$ on the generic fiber $(V,D)$ of $(\cV,\cD)$. Extending a result of Silverman \cite{Sil83} for the variation of the usual canonical height in algebraic families of abelian varieties, Call and Silverman \cite[Theorem 4.1]{Call-Silverman} proved that 
\begin{equation}
\label{C-S result}
\hhat_{\cV_\l,\cD_\l,\Phi_\l}(P_\l)=\hhat_{V,D, \Phi}(P)\cdot h(\l)+o(h(\l)),
\end{equation}
where $h(\l)$ is a Weil height on $X$. In the special case $\cV\lra X$ is an elliptic surface, Tate \cite{tate83} improved the error term of \eqref{C-S result} to $O(1)$ (where the implied constant depends on $P$ only, and it is independent of $\l$). Furthermore, Silverman \cite{Silverman83, Silverman-2, Silverman-3} proved that the difference of the main terms from \eqref{C-S result}, in addition to being bounded, varies quite regularly as a function
of $\l$, breaking up into a finite sum of well-behaved functions at finitely many places. It is natural to ask whether there are other instances when the error term of \eqref{C-S result} can be improved to $O_P(1)$. 

In \cite{ingram10}, Ingram showed that when $\Phi_\l$ is an algebraic family of polynomials acting on the affine line, then again the error term in \eqref{C-S result} is $O(1)$ (when the parameter space $X$ is the projective line). More precisely, Ingram proved that for an arbitrary parameter curve $X$, there exists   $D = D(\bff, P) \in \pic(X)\otimes \Q$ of
degree $\hhat_{\bff}(P)$ such that $
 \hhat_{\bff_\l}(P_\l) = h_D(\l) + O(1)$. 
This result is an analogue of Tate's theorem~\cite{tate83} in  the
setting of arithmetic dynamics. Using this result and applying an
observation of Lang~\cite[Chap.~5, Prop.~5.4]{Lang-diophantine}, the error term  can be improved to $O(h(\l)^{1/2})$ and 
furthermore, in the special case where $X = \bP^1$ the error term
can be replaced by  $O(1)$.  In \cite{prep}, Ghioca, Hsia an Tucker showed that the error term is also uniformly bounded independent of $\l\in X$ (an arbitrary projective curve) when $\Phi_\l$ is an algebraic family of rational maps satisfying the properties:
\begin{enumerate}
\item[(a)] each $\Phi_\l$ is superattracting at infinity, i.e. if $\Phi_\l=\frac{P_\l}{Q_\l}$ for algebraic families of polynomials $P_\l,Q_\l\in \Qbar[z]$, then $\deg(P_\l)\ge 2+\deg(Q_\l)$; and 
\item[(b)] the resultant of $P_\l$ and $Q_\l$ is a nonzero constant.
\end{enumerate}
The condition (a) is very mild for applications; on the other hand condition (b) is restrictive. Essentially condition (b) asks that $\Phi_\l$ is a well-defined rational map of same degree as on the generic fiber, i.e., all fibers of $\Phi$ are \emph{good}.

Our main result is to improve the error term of \eqref{C-S result} to $O(1)$ for the algebraic family of rational maps $\bff_t(z)=\frac{z^d+t}{z}$ where the parameter $t$ varies on the projective line. We denote by $\hhat_{\bff_\l}$ the canonical height associated to $\bff_\l$ for each $t=\l\in\Qbar$, and we denote by $\hhat_\bff$ the canonical height on the generic fiber (i.e., with respect to the map $\bff_t(z):=\frac{z^d+t}{z}\in\Qbar(t)(z)$).
\begin{thm}
\label{variation of canonical height}
Let $\bfc\in\Qbar(t)$ be a rational map, let $d\ge 2$ be an integer, and let $\{\bff_t\}$ be the algebraic family of rational maps given by $\bff_t(z):=\frac{z^d+t}{z}$. Then as  $t=\l$ varies in $\Qbar$ we have
\begin{equation}
\label{formula variation}
\hhat_{\bff_\l}(\bfc(\l)) = \hhat_\bff(\bfc)\cdot h(\l) + O(1),
\end{equation}
where the constant in $O(1)$ depends only on $\bfc$, and it is independent of  $\l$.
\end{thm}
Alternatively, Theorem~\ref{variation of canonical height} yields that the function $\l\mapsto \hhat_{\bff_\l}(\bfc(\l))$ is a Weil height on $\bP^1$ associated to the divisor $\hhat_\bff(\bfc)\cdot \infty\in\pic(\bP^1)\otimes\Q$. 

We note that on the fiber $\l=0$, the corresponding rational map $\Phi_0$ has degree $d-1$ rather than $d$ (which is the generic degree in the family $\Phi_\l$). So, our result is the \emph{first} example of an algebraic family of rational maps (which are neither totally ramified at infinity, nor Latt\'es maps, and also admit \emph{bad} fibers) for which the error term in \eqref{C-S result} is $O(1)$. In addition, we mention that the family $\bff_t(z)=\frac{z^d+t}{z}$ for $t\in\C$ is interesting also from the complex dynamics point of view. Devaney and Morabito \cite{Devaney} proved that the Julia sets $\{J_t\}_{t\in\C}$ of the above maps converge to the unit disk as $t$ converges to $0$ along the rays ${\rm Arg}(t)=\frac{(2k+1)\pi}{d-1}$ for $k=0,\dots, d-1$, providing thus an example of a family of rational maps whose Julia sets have empty interior, but in the limit, these sets converge to a set with nonempty interior.

A special case of our Theorem~\ref{variation of canonical height} is when the starting point $\bfc$ is constant; in this case we can give a precise formula for the $O(1)$-constant appearing in \eqref{formula variation}.
\begin{thm}
\label{precise constant}
Let $d\ge 2$ be an integer, let $\alpha$ be an algebraic number, let $K=\Q(\alpha)$ and let $\ell$ be the number of  non-archimedean places $|\cdot |_v$ of $K$ satisfying $|\alpha|_v\notin\{ 0,1\}$. If $\{\bff_t\}$ is the algebraic family of rational maps given by $\bff_t(z):=\frac{z^d+t}{z}$, then
$$\left|\hhat_{f_\l}(\alpha) - \hhat_\bff(\alpha)\cdot h(\l)\right| < 3d\cdot (1+\ell+2h(\alpha)),$$
as $t=\l$ varies in $\Qbar$.  
\end{thm}
In particular, Theorem~\ref{precise constant} yields an effective way for determining for any given $\alpha\in\Qbar$ the set of parameters $\l$ contained in a number field of bounded degree such that $\alpha$ is preperiodic for $\bff_\l$. Indeed, if $\alpha\in\Qbar$ then either $\alpha=0$ and then it is preperiodic for all $\bff_\l$, or $\alpha\ne 0$ in which case generically $\alpha$ is not preperiodic and $\hhat_{\bff}(\alpha)=\frac{1}{d}$ (see Proposition~\ref{canonical height generic nonzero}). So, if $\alpha\in\Qbar^*$ 
is preperiodic for $\bff_\l$ then $\hhat_{\bff_\l}(\alpha)=0$ and thus, Theorem~\ref{precise constant} yields that 
\begin{equation}
\label{bounded parameter}
h(\l)<3d^2\cdot (1+\ell+2h(\alpha)). 
\end{equation}
For example, if $\alpha$ is a root of unity, then $h(\l)<3d^2$ for all parameters $\l\in\qbar$ such that $\alpha$ is preperiodic for $\bff_\l$.

Besides the intrinsic interest in studying the above problem, recently it was discovered a very interesting connection between the variation of the canonical height in algebraic families and the problem of unlikely intersections in algebraic dynamics (for a beautiful introduction to this area, please see the book of Zannier \cite{Zannier}). Masser and Zannier \cite{M-Z-1, M-Z-2} proved that for the family of Latt\'es maps $\bff_\l(z)=\frac{(z^2-\l)^2}{4z(z-1)(z-\l)}$ there exist at most finitely many $\l\in\Qbar$ such that both $2$ and $3$ are preperiodic for $\bff_\l$. Geometrically, their result says the following: given the Legendre family of elliptic curves $E_\l$ given by the equation $y^2=x(x-1)(x-\l)$, there exist at most finitely many $\l\in\Qbar$ such that $P_\l:=\left(2,\sqrt{2(2-\l)}\right)$ and $Q_\l:=\left(3, \sqrt{6(3-\l)}\right)$ are simultaneously  torsion points for $E_\l$. Later Masser and Zannier \cite{M-Z-3} extended their result by proving that for any two sections $P_\l$ and $Q_\l$ on any elliptic surface $E_\l$, if there exist infinitely many $\l\in\C$ such that both $P_\l$ and $Q_\l$ are torsion for $E_\l$ then the two sections are linearly dependent over $\Z$. Their proof uses the recent breakthrough results of Pila and Zannier \cite{P-Z}. Moreover, Masser and Zannier exploit in a crucial way the existence of the analytic uniformization map for elliptic curves.   
Motivated by a question of Zannier, Baker and DeMarco \cite{Baker-Demarco} showed that for any $a,b\in\C$, if there exist infinitely many $\l\in\C$ such that both $a$ and $b$ are preperiodic for $\bff_\l(z)=z^d+\l$ (where $d\ge 2$), then $a^d=b^d$. Later their result was generalized by Ghioca, Hsia and Tucker \cite{prep-0} to arbitrary families of polynomials. The method of proof employed in both \cite{Baker-Demarco} and \cite{prep-0} uses an equidistribution statement (see \cite[Theorem~7.52]{Baker-Rumely} and \cite{Favre-Rivera-0, Favre-Rivera}) for points of small canonical height on Berkovich spaces. Later, using the powerful results of Yuan and Zhang \cite{Yuan, Yuan-Zhang} on metrized line bundles, Ghioca, Hsia and Tucker \cite{prep} proved the first results on unlikely intersections  for families of rational maps and also for families of endomorphisms of higher dimensional projective spaces. 
The difference between the results of \cite{Baker-Demarco, prep-0, prep} and the results of \cite{M-Z-1, M-Z-2, M-Z-3} is that for arbitrary families of polynomials there is no analytic uniformization map as in the case of the elliptic curves. Instead one needs to employ a more careful analysis of the local canonical heights associated to the family of rational maps. This led the authors of \cite{prep} to prove the error term in \eqref{C-S result} is $O(1)$ for the rational maps satisfying conditions $(a)-(b)$ listed above. Essentially, in order to use the equidistribution results of Baker-Rumely, Favre-Rivera-Letelier, and Yuan-Zhang, one needs to show that certain metrics converge uniformly and in turn this relies on showing that the local canonical heights associated to the corresponding family of rational maps vary uniformly across the various fibers of the family; this leads to improving to $O(1)$ the error term in \eqref{C-S result}. It is of great interest to see whether the results on unlikely intersections can be extended to more general families of rational maps beyond families of Latt\'es maps \cite{M-Z-1, M-Z-2, M-Z-3}, or of polynomials \cite{Baker-Demarco, prep-0}, or of rational maps with good fibers for all points in the parameter space \cite{prep}. 
On the other hand, a preliminary step to ensure  the strategy from \cite{Baker-Demarco, prep, prep-0} can be employed to proving new results on unlikely intersections in arithmetic dynamics is to improve to $O(1)$ the error term from \eqref{C-S result}. For example, using the exact strategy employed in \cite{prep}, the results of our paper yield that if $\bfc_1(t),\bfc_2(t)\in \Qbar(t)$ have the property that there exist infinitely many $\l\in\Qbar$ such that both $\bfc_1(\l)$ and $\bfc_2(\l)$ are preperiodic under the action of $\bff_\l(z):= \frac{z^d+\l}{z}$, then for \emph{each} $\l\in\Qbar$ we have that $\bfc_1(\l)$ is preperiodic for $\bff_\l$ if and only if $\bfc_2(\l)$ is preperiodic for $\bff_\l$. Furthermore, if in addition $\bfc_1,\bfc_2$ are constant, then the same argument as in \cite{prep} yields that for \emph{each} $\l\in\Qbar$, we have  $\hhat_{\bff_\l}(\bfc_1)=\hhat_{\bff_\l}(\bfc_2)$. Finally, this condition should yield that $\bfc_1=\bfc_2$; however finding the exact relation between $\bfc_1$ and $\bfc_2$ is usually difficult (see the discussion from \cite{prep-0, prep}).

In our proofs we use in an essential way the decomposition of the (canonical) height in a sum of local (canonical) heights. So, in order 
 to prove Theorems~\ref{variation of canonical height} and \ref{precise constant} we show first (see Proposition~\ref{each place}) that for all but finitely many places $v$, the contribution of the corresponding local height to $d^2\cdot \hhat_{\bff_\l}(\bfc(\l))$ matches the $v$-adic contribution to the height for the second iterate $\bff^2_\l(\bfc(\l))$. This allows us to conclude  that $\left|\hhat_{\bff_\l}(\bfc(\l))-\frac{h(\bff_\l^2(\bfc(\l)))}{d^2}\right|$ is uniformly bounded as $\l$ varies. Then, using that $\deg_\l(\bff^2_\l(\bfc(\l)))=\hhat_\bff(\bfc)\cdot d^2$, an application of the height machine finishes our proof. The main difficulty lies in proving that for each place $v$ the corresponding local contribution to $d^2\cdot \hhat_{\bff_\l}(\bfc(\l))$ varies from the $v$-adic contribution to $h(\bff^2_\l(\bfc(\l)))$ by an amount bounded solely in terms of $v$ and of  $\bfc$. In order to derive our conclusion we first prove the statement for the special case when $\bfc$ is constant. Actually, in this latter case we can prove (see Propositions~\ref{e=0} and \ref{e=0 d=2})  that $ \left|\hhat_{\bff_\l}(\bfc(\l))-\frac{h(\bff_\l(\bfc(\l)))}{d}\right|$ is uniformly bounded as $\l$ varies. Then for the general case of Proposition~\ref{each place}, we apply Propositions~\ref{e=0} and \ref{e=0 d=2} to the first iterate of $\bfc(\l)$ under $\bff_\l$. For our analysis, we split the proof into $3$ cases:
\begin{enumerate}
\item[(i)] $|\l|_v$ is much larger than the absolute values of the coefficients of the polynomials $A(t)$ and $B(t)$ defining $\bfc(t):=\frac{A(t)}{B(t)}$. 
\item[(ii)] $|\l|_v$ is bounded above and below by constants depending only on the absolute values of the coefficients of $A(t)$ and of $B(t)$. 
\item[(iii)] $|\l|_v$ is very small. 
\end{enumerate}
The cases (i)-(ii) are not very difficult and the same proof is likely to work for more general families of rational maps (especially if $\infty$ is a superattracting point for the rational maps $\bff_\l$; note that the case $d=2$ for Theorems~\ref{variation of canonical height} and \ref{precise constant} requires a different approach). However, case (iii) is much harder, and here we use in an essential way the general form of our family of maps. It is not  surprising that this is the hard case since $\l=0$ is the only bad fiber of the family $\bff_\l$.
We do not know whether the error term of $O(1)$ can be obtained for the variation of the canonical height in more general families of rational maps. It seems  that each time $\l$ is close to a singularity of the family (i.e., $\l$ is close $v$-adically to some $\l_0$ for which $\deg(\Phi_{\l_0})$ is less than the generic degree in the family) would require a different approach.

The plan of our paper is as follows. In  the next section we setup the notation for our paper. Then in Section~\ref{generic fiber section} we compute the height $\hhat_\bff(\bfc)$ on the generic fiber of our dynamical system. We continue in Section~\ref{reductions section} with a series of reductions of our main results; we reduce Theorem~\ref{variation of canonical height} to proving Proposition~\ref{each place}. We conclude by proving Theorem~\ref{precise constant} in Section~\ref{section e=0}, and then finishing the proof of Proposition~\ref{each place}  in Section~\ref{proofs}.

\medskip

\emph{Acknowledgments.} We thank Joseph Silverman and Laura DeMarco for useful discussions regarding this project.


\section{Notation}
\label{notation}


\subsection{Generalities}

For a rational function $f(z)$, we denote by $f^n(z)$ its $n$-th iterate (for any $n\ge 0$, where $f^0$ is the identity map). We call a point $P$ preperiodic if its orbit under $f$ is finite.

For each real number $x$, we denote $\log^+x:=\log\max\{1,x\}$. 


\subsection{Good reduction for rational maps}

Let $K$ be a number field, let $v$ be a nonarchimedean valuation on
$K$, let $\fO_v$ be the ring of $v$-adic integers of $K$,
and let $k_v$ be the residue field at $v$. If $A,B\in K[z]$ are coprime polynomials, then $\varphi(z):=A(z)/B(z)$ has good reduction (see \cite{Silverman07}) at all places $v$ satisfying the properties:
\begin{enumerate}
\item[(1)] the coefficients of $A$ and of $B$ are in $\fO_v$;
\item[(2)] the leading coefficients of $A$ and of $B$ are units in $\fO_v$; and 
\item[(3)] the resultant of the polynomials $A$ and $B$ is a unit in $\fO_v$. 
\end{enumerate}
Clearly, all but finitely many places $v$ of $K$ satisfy the above conditions (1)-(3). In particular this yields that if we reduce modulo $v$ the coefficients of both $A$ and $B$, then the induced rational map $\overline{\varphi}(z):=\overline{A}(z)/\overline{B}(z)$ is a well-defined rational map defined over $k_v$ of same degree as $\varphi$.


\subsection{Absolute values}

We denote by $\Omega_\Q$ the set of all (inequivalent) absolute values of $\Q$ with the usual normalization so that the product formula holds: $\prod_{v\in\Omega_\Q}|x|_v=1$ for each nonzero $x\in\Q$. For each $v\in\Omega_\Q$, we fix an extension of $|\cdot |_v$ to $\Qbar$. 


\subsection{Heights}
\label{heights subsection}
\subsubsection{Number fields}

Let $K$ be a number field. For each $n\ge 1$, if $P:=[x_0:\cdots :x_n]\in \bP^n(K)$ then the Weil height of $P$ is
$$h(P):=\frac{1}{[K:\Q]}\cdot \sum_{\sigma:K\lra\Qbar}\sum_{v\in\Omega_\Q} \log\max\{|\sigma(x_0)|_v,\cdots ,|\sigma(x_n)|_v\},$$
where the first summation runs over all embeddings $\sigma:K\lra \Qbar$. 
The definition is independent of the choice of coordinates $x_i$ representing $P$ (by an application of the product formula) and it is also independent of the particular choice of number field $K$ containing the coordinates $x_i$ (by the fact that each place $v\in\Omega_\Q$ is defectless, as defined by \cite{Serre}). In this paper we will be concerned mainly with the height of points in $\bP^1$; furthermore, if $x\in \Qbar$, then we identify $x$ with $[x:1]\in\bP^1$ and define its height accordingly. The basic properties for heights which we will use are: for all $x,y\in\Qbar$ we have
\begin{enumerate}
\item[(1)] $h(x+y)\le h(x)+h(y)+\log(2)$,  
\item[(2)] $h(xy)\le h(x)+h(y)$, and
\item[(3)] $h(1/x)=h(x)$.
\end{enumerate}

\subsubsection{Function fields}
We will also work with the height of rational functions (over $\Qbar$). So, if $L$ is any field, then the Weil height of a rational function $g\in L(t)$ is defined to be its degree.


\subsection{Canonical heights}
\subsubsection{Number fields}

Let $K$ be a number field, and let $f\in K(z)$ be a rational map of degree $d\ge 2$. Following \cite{Call-Silverman} we define the canonical height of a point $x\in \bP^1(\Qbar)$ as 
\begin{equation}
\label{definition of canonical height}
\hhat_f(x)=\lim_{n\to\infty} \frac{h(f^n(x))}{d^n}.
\end{equation}
As proved in \cite{Call-Silverman}, the difference $|h(x )-\hhat_f(x )|$ is uniformly bounded for all $x\in \bP^1(\Qbar)$. Also, $\hhat_f(x)=0$ if and only if $x$ is a preperiodic point for $f$. If $x\in \Qbar$ then we view it embedded in $\bP^1$ as $[x:1]$ and denote by $\hhat_f(x)$ its canonical height under $f$ constructed as above.

\subsubsection{Function fields}

Let $L$ be an arbitrary field, let $f\in L(t)(z)$ be a rational function of degree $d\ge 2$, and let $x\in L(t)$. Then the canonical height $\hhat_f(x):=\hhat_f([x:1])$ is defined the same as in \eqref{definition of canonical height}.


\subsection{Canonical heights for points and rational maps as they vary in algebraic families}
\subsubsection{Number fields}

If $\l\in \Qbar^*$, 
$x=[A:B]\in \bP^1(\Qbar)$ and $f_\l(z):=\frac{z^d+\l}{z}$, then we can define $\hhat_{f_\l}(x)$  alternatively as follows. We let $A_{\l,[A:B],0}:=A$ and $B_{\l,[A:B],0}:=B$, and for each $n\ge 0$ we let
$$A_{\l,[A:B],n+1}:=A_{\l,[A:B],n}^d+\l\cdot B_{\l,[A:B],n}^d\text{ and }B_{\l,[A:B],n+1}:=A_{\l,[A:B],n}\cdot B_{\l,[A:B],n}^{d-1}.$$
Then $f_\l^n([A:B])=[A_{\l,[A:B],n}:B_{\l,[A:B],n}]$ and so, 
$$\hhat_{f_\l}(x)=\lim_{n\to\infty} \frac{h([A_{\l,[A:B],n}:B_{\l,[A:B],n}])}{d^n}.$$
Also, for each place $v$, we define the local canonical height of $x=[A:B]$ with respect to $f_\l$ as 
\begin{equation}
\label{defi local canonical height 0}
\hhat_{f_\l,v}(x)=\lim_{n\to\infty} \frac{\log\max\left\{\left|A_{\l,[A:B],n}\right|_v, \left|B_{\l,[A:B],n}\right|_v\right\}}{d^n}.
\end{equation}

If $x\in \Qbar$ we view it embedded in $\bP^1(\Qbar)$ as $[x:1]$ and compute its canonical heights (both global and local) under $f_\l$ as above starting with $A_{\l,x,0}:=x$ and $B_{\l,x,0}:=1$. 

For $x=A/B$ with $B\ne 0$, we get that 
\begin{equation}
\label{conversion 0}
A_{\l,[A:B],n}=A_{\l,x,n}\cdot B^{d^n}\text{ and }B_{\l,[A:B],n}=B_{\l,x,n}\cdot B^{d^n},
\end{equation}
for all $n\ge 0$. If in addition $A\ne 0$, then $B_{\l,[A:B],1}=A\cdot B^{d-1}\ne 0$ and then for all $n\ge 0$ we have
\begin{equation}
\label{conversion -1}
A_{\l,[A:B],n+1}=A_{\l,f_\l(x),n}\cdot B_{\l,[A:B],1}^{d^n}\text{ and }B_{\l,[A:B],n+1}=B_{\l,f_\l(x),n}\cdot B_{\l,[A:B],1}^{d^n}
\end{equation}
and in general, if $B_{\l,[A:B],k_0}\ne 0$, then
\begin{equation}
\label{conversion k_0}
A_{\l,[A:B],n+k_0}=A_{\l,f_\l^{k_0}(x),n}\cdot B_{\l,[A:B],k_0}^{d^n}\text{ and }B_{\l,[A:B],n+k_0}=B_{\l,f_\l^{k_0}(x),n}\cdot B_{\l,[A:B],k_0}^{d^n}.
\end{equation}

We will be interested also in studying the variation of the canonical height of a family of starting points parametrized by a rational map (in $t$) under the family $\{\bff_t(z)\}$ of rational maps. As before, $\bff_t(z):=\frac{z^d+t}{z}$, and for each $t=\l\in\Qbar$ we get a map in the above family of rational maps. When we want to emphasize the fact that each $\bff_\l$ (for $\l\in\Qbar^*$) belongs to this family of rational maps (rather than being a single rational map), we will use the boldface letter $\bff$ instead of $f$. Also we let $\bfc(t):=\frac{\bfA(t)}{\bfB(t)}$ where $\bfA,\bfB\in K[t]$ are coprime polynomials defined over a number field $K$. Again, for each $t=\l\in\Qbar$ we get a point $\bfc(\l)\in \bP^1(\Qbar)$. 

We define $\bfA_{\bfc,n}(t)\in K[t]$ and $\bfB_{\bfc,n}(t)\in K[t]$  so that for each $n\ge 0$ we have $\bff_t^n(\bfc(t))=[\bfA_{\bfc, n}(t):\bfB_{\bfc, n}(t)]$. In particular, for each $t=\l\in\Qbar$ we have $\bff_\l^n(\bfc(\l))=[\bfA_{\bfc, n}(\l):\bfB_{\bfc, n}(\l)]$.

We let $\bfA_{\bfc, 0}(t):=\bfA(t)$ and $\bfB_{\bfc,0}(t):=\bfB(t)$. Our definition for $\bfA_{\bfc,n}$ and $\bfB_{\bfc,n}$ for $n= 1$ will depend on whether $\bfA(0)$ (or equivalently $\bfc(0)$) equals $0$ or not. If $\bfA(0)\ne 0$, then we define
\begin{equation}
\label{1st iterate nonzero}
\bfA_{\bfc,1}(t):=\bfA(t)^d+t \bfB(t)^d\text{ and }\bfB_{\bfc,1}(t):=\bfA(t)\bfB(t)^{d-1},
\end{equation}
while if $\bfc(0)=0$, then
\begin{equation}
\label{1st iterate zero}
\bfA_{\bfc,1}(t):=\frac{\bfA(t)^d+t \bfB(t)^d}{t}\text{ and }\bfB_{\bfc,1}(t):=\frac{\bfA(t)\bfB(t)^{d-1}}{t}.
\end{equation}
Then for each positive integer $n$ we let
\begin{equation}
\label{general iterate defi}
\bfA_{\bfc,n+1}(t):=\bfA_{\bfc,n}(t)^d+t\cdot \bfB_{\bfc, n}(t)^d\text{ and }\bfB_{\bfc, n+1}(t):=\bfA_{\bfc, n}(t)\cdot \bfB_{\bfc, n}(t)^{d-1}.
\end{equation}
Whenever it is clear from the context, we will use $\bfA_n$ and $\bfB_n$ instead of $\bfA_{\bfc, n}$ and $\bfB_{\bfc, n}$ respectively. For each $t=\l\in\Qbar$, the canonical height of $\bfc(\l)$ under the action of $\bff_\l$ may be computed as follows:
$$\hhat_{\bff_\l}(\bfc(\l))=\lim_{n\to\infty}\frac{h\left([\bfA_{\bfc,n}(\l): \bfB_{\bfc,n}(\l)]\right)}{d^n}.$$
Also, for each place $v$, we define the local canonical height of $\bfc(\l)$ at $v$ as follows:
\begin{equation}
\label{defi local canonical height}
\hhat_{\bff_\l,v}(\bfc(\l)) = \lim_{n\to\infty} \frac{\log\max\{\left|\bfA_{\bfc,n}(\l)\right|_v, \left|\bfB_{\bfc,n}(\l)\right|_v\}}{d^n}.
\end{equation}
The limit in \eqref{defi local canonical height} exists, as proven in Corollary~\ref{the limit exists}.

The following is a simple observation based on \eqref{conversion k_0}: if $\l\in\Qbar$ such that $\bfB_{\bfc, k_0}(\l)\ne 0$, then for each $k_0,n\ge 0$ we have
\begin{equation}
\label{conversion}
\bfA_{\bfc, n+k_0}(\l)=\bfB_{\bfc, k_0}(\l)^{d^n}\cdot A_{\l, f^{k_0}_\l(\bfc(\l)), n}\text{ and }\bfB_{\bfc, n+k_0}(\l)=\bfB_{\bfc, k_0}(\l)^{d^n}\cdot B_{\l,f^{k_0}_\l(\bfc(\l)), n}.
\end{equation}

\subsubsection{Function fields}

We also compute the canonical height of $\bfc(t)$ on the generic fiber of the family of rational maps $\bff$ with respect to the action of $\bff_t(z)=\frac{z^d+t}{z}\in \Q(t)(z)$ as follows
$$\hhat_{\bff}(\bfc):=\hhat_{\bff_t}(\bfc(t)):=\lim_{n\to\infty} \frac{h(\bff_t^n(\bfc(t)))}{d^n} =\lim_{n\to\infty} \frac{\deg_t(\bff_t^n(\bfc(t)))}{d^n}.$$


\section{Canonical height on the generic fiber}
\label{generic fiber section}

For each $n\ge 0$, the map $t\lra \bff_t^n(\bfc(t))$ is a rational map; so, $\deg(\bff^n_t(\bfc(t)))$ will always denote its degree. Similarly, letting $\bff(z):=\frac{z^d+t}{z}\in\qbar(t)(z)$ and $\bfc(t):=\frac{\bfA(t)}{\bfB(t)}$ for coprime polynomials $\bfA,\bfB\in\Qbar[t]$, then $\bff^n(\bfc(t))$ is a rational function for each $n\ge 0$. In this section we compute $\hhat_\bff(\bfc)$. It is easier to split the proof into two cases depending on whether $\bfc(0)=0$ (or equivalently $\bfA(0)=0$) or not. 

\begin{prop}
\label{canonical height generic nonzero}
If $\bfc(0)\ne 0$, then 
$$\hhat_\bff(\bfc)=\frac{\deg(\bff_t(\bfc(t)))}{d}= \frac{\deg(\bff^2_t(\bfc(t)))}{d^2}.$$
\end{prop}

\begin{proof}
According to \eqref{1st iterate nonzero} and \eqref{general iterate defi} we have defined  $\bfA_{\bfc,n}(t)$ and $\bfB_{\bfc,n}(t)$ in this case. 
It is easy to prove that $\deg(\bfA_n)>\deg(\bfB_n)$ for all positive integers $n$. Indeed, if $\deg(\bfA)> \deg(\bfB)$, then an easy induction yields that $\deg(\bfA_n)> \deg(\bfB_n)$ for all $n\ge 0$. If $\deg(\bfA)\le \deg(\bfB)$, then $\deg(\bfA_1)=1+d\cdot \deg(\bfB)> d\cdot \deg(\bfB)\ge \deg(\bfB_1)$. Again an easy induction finishes the proof that $\deg(\bfA_n)>\deg(\bfB_n)$ for all $n\ge 1$.

In particular, we get that $\deg(\bfA_n)=d^{n-1}\cdot \deg(\bfA_1)$ for all $n\ge 1$. The following claim will finish our proof.  
\begin{claim}
\label{A and B coprime}
For each $n\ge 0$, $\bfA_n$ and $\bfB_n$ are coprime.
\end{claim}

\begin{proof}[Proof of Claim~\ref{A and B coprime}.]
The statement is true for $n=0$ by definition. Assume now that it holds for all $n\le N$ and we'll show that $\bfA_{N+1}$ and $\bfB_{N+1}$ are coprime. 

Assume there exists $\alpha\in\Qbar$ such that the polynomial $t-\alpha$ divides both $\bfA_{N+1}(t)$ and $\bfB_{N+1}(t)$. First we claim that $\alpha\ne 0$. Indeed, if $t$ would divide $\bfA_{N+1}$, then it would also divide $\bfA_N$ and inductively we would get that $t\mid \bfA_0(t)=\bfA(t)$, which is a contradiction since $A(0)\ne 0$. So, indeed $\alpha\ne 0$. But then from the fact that both $\bfA_{N+1}(\alpha)=0=\bfB_{N+1}(\alpha)$ (and $\alpha\ne 0$) we obtain from the recursive formula defining $\{\bfA_n\}_n$ and $\{\bfB_n\}_n$ that also $\bfA_N(\alpha)=0$ and $\bfB_N(\alpha)=0$. However this contradicts the assumption that $\bfA_N$ and $\bfB_N$ are coprime. Thus $\bfA_n$ and $\bfB_n$ are coprime for all $n\ge 0$. 
\end{proof}

Using the definition of $\hhat_\bff(\bfc)$ we conclude the proof of Proposition~\ref{canonical height generic nonzero}.
\end{proof}

If $\bfc(0)=0$ (or equivalently $\bfA(0)=0$) the proof is very similar, only that this time we use \eqref{1st iterate zero} to define $\bfA_1$ and $\bfB_1$.
\begin{prop}
\label{canonical height generic zero}
If $\bfc(0)= 0$, then 
$$\hhat_\bff(\bfc)=\frac{\deg(\bff_t(\bfc(t)))}{d}= \frac{\deg(\bff^2_t(\bfc(t)))}{d^2}.$$
\end{prop}

\begin{proof}
Since $t\mid \bfA(t)$ we obtain that $\bfA_1,\bfB_1\in\Qbar[t]$; moreover, they are coprime because $\bfA$ and $\bfB$ are coprime. Indeed, $t$ does not divide $\bfB(t)$ and so, because $t$ divides $\bfA(t)$ and $d\ge 2$, we conclude that $t$ does not divide $\bfA_1(t)$. Now, if there exists some $\alpha\in\Qbar^*$ such that both $\bfA_1(\alpha)=\bfB_1(\alpha)=0$, then we obtain that also both $\bfA(\alpha)=\bfB(\alpha)=0$, which is a contradiction. 

Using that $\bfA_1$ and $\bfB_1$ are coprime, and also that $t\nmid \bfA_1$, the same reasoning as in the proof of Claim~\ref{A and B coprime} yields that $\bfA_n$ and $\bfB_n$ are coprime for each $n\ge 1$. 

Also, arguing as in the proof of Proposition~\ref{canonical height generic nonzero}, we obtain that $\deg(\bfA_n)> \deg(\bfB_n)$ for all $n\ge 1$. Hence, $$\deg(\bff_t^n(\bfc(t)))=\deg_t(\bfA_n(t))=d^{n-2}\cdot \deg(\bff_t^2(\bfc(t)))= d^{n-1}\cdot \deg(\bff_t(\bfc(t))),$$
as desired.
\end{proof}


\section{Reductions}
\label{reductions section}

With the above notation, Theorem~\ref{variation of canonical height} is equivalent with showing that
\begin{equation}
\label{restatement}
\lim_{n\to\infty} \frac{h([\bfA_{\bfc,n}(\l):\bfB_{\bfc,n}(\l)])}{d^n} = \hhat_{\bff}(\bfc)\cdot h(\l)+O_\bfc(1).
\end{equation}

In all of our arguments we assume $\l\ne 0$, and also that $\bfA$ and $\bfB$ are not identically equal to $0$ (where $\bfc=\bfA/\bfB$ with $\bfA,\bfB\in \Qbar[t]$ coprime). Obviously excluding the case $\l=0$ does not affect the validity of Theorem~\ref{variation of canonical height} (the quantity $\hhat_{\bff_0}(\bfc(0))$ can be absorbed into the $O(1)$-constant). In particular, if $\l\ne 0$ then the definition of $\bfA_{\bfc,1}$ and $\bfB_{\bfc,1}$ (when $\bfc(0)=0$) makes sense (i.e. we are allowed to divide by $\l$). Also, if $\bfA$ or $\bfB$ equal $0$ identically, then $\bfc(\l)$ is preperiodic for $\bff_\l$ for all $\l$ and then again Theorem~\ref{variation of canonical height} holds trivially.

\begin{prop}
\label{almost all places are fine}
Let $\l\in\Qbar^*$. Then for all but finitely many $v\in\Omega_\Q$, we have $\log\max\{|\bfA_{\bfc ,n}(\l)|_v, |\bfB_{\bfc ,n}(\l)|_v\}=0$ for all $n\in\bN$.
\end{prop}

\begin{proof}
First of all, for the sake of simplifying our notation (and noting that $\bfc$ and $\l$ are fixed in this Proposition), we let $A_n:=\bfA_{\bfc,n}(\l)$ and $B_n:=\bfB_{\bfc, n}(\l)$. 

From the definition of $A_1$ and $B_1$ we see that not both are equal to $0$ (here we use also the fact that $\l\ne 0$ which yields that if both $A_1$ and $B_1$ are equal to $0$ then $\bfA(\l)=\bfB(\l)=0$, and this contradicts the fact that $\bfA$ and $\bfB$ are coprime).  
Let $S$ be the set of all non-archimedean places $v\in\Omega_\Q$ such that $|\l|_v=1$ and also $\max\{|A_1|_v, |B_1|_v\}=1$. Since  not both $A_1$ and $B_1$ equal $0$ (and also $\l\ne 0$), then all but finitely many non-archimedean places $v$ satisfy the above conditions. 

\begin{claim}
\label{useful claim}
If $v\in S$, then $\max\{|A_n|_v, |B_n|_v\}=1$ for all $n\in\bN$. 
\end{claim}

\begin{proof}[Proof of Claim~\ref{useful claim}.]
This claim follows easily by induction on $n$; the case $n=1$ follows by the definition of $S$. Since 
$$\max\{|A_n|_v,|B_n|_v\}=1$$ 
and $|\l|_v= 1$ 
then $\max\{|A_{n+1}|_v, |B_{n+1}|_v\}\le 1$. Now, if $|A_n|_v=|B_n|_v=1$ then $|B_{n+1}|_v=1$. On the other hand, if $\max\{|A_n|_v,|B_n|_v\}=1>\min\{|A_n|_v, |B_n|_v\}$, then $|A_{n+1}|_v=1$ (because $|\l|_v=1$).
\end{proof}
Claim~\ref{useful claim} finishes the proof of Proposition~\ref{almost all places are fine}.
\end{proof}

We let $K$ be the finite extension of $\Q$ obtained by adjoining the coefficients of both $\bfA$ and $\bfB$ (we recall that $\bfc(t)=\bfA(t)/\bfB(t)$). Then $\bfA_n(\l):=\bfA_{\bfc, n}(\l)\text{, }\bfB_n(\l):=\bfB_{\bfc, n}(\l)\in K(\l)$ for each $n$ and for each $\l$. 
Proposition~\ref{almost all places are fine} allows us to invert the limit from the left-hand side of \eqref{restatement} with the following sum
$$h([\bfA_n(\l):\bfB_n(\l)])=\frac{1}{[K(\l):\Q]}\cdot \sum_{\sigma:K(\l)\lra \Qbar}\sum_{v\in\Omega_\Q} \log\max\{|\sigma(\bfA_n(\l))|_v, |\sigma(\bfB_n(\l))|_v\},$$
because for all but finitely many places $v$, we have $\log\max\{|\sigma(\bfA_n(\l))|_v, |\sigma(\bfB_n(\l))|_v\}=0$. Also we note that $\sigma(\bfA_{\bfc,n}(\l))=\bfA_{\bfc^\sigma,n}(\sigma(\l))$ and $\sigma(\bfB_{\bfc ,n}(\l))=\bfB_{\bfc^\sigma,n}(\sigma(\l))$, where $\bfc^\sigma$ is the polynomial whose coefficients are obtained by applying the homomorphism $\sigma\in\Gal(\Qbar/\Q)$ to each coefficient of $\bfc$.  Using the definition of the local canonical height from \eqref{defi local canonical height}, we observe that \eqref{restatement} is equivalent with showing that
\begin{equation}
\label{restatement 2}
 \frac{1}{[K(\l):\Q]} \sum_{v\in\Omega_\Q}\sum_{\sigma:K(\l)\lra \Qbar}  \hhat_{\bff_{\sigma(\l)},v}(\bfc^{\sigma}(\sigma(\l))) = \hhat_\bff(\bfc) h(\l)+O_\bfc(1).
\end{equation}
For each $v\in\Omega_\Q$, and each $n\ge 0$ we let $M_{\bfc,n,v}(\l):=\max\{|\bfA_{\bfc,n}(\l)|_v, |\bfB_{\bfc,n}(\l)|_v\}$. When $\bfc$ is fixed, we will use the notation $M_{n,v}(\l):=M_{\bfc,n,v}(\l)$; if $\l$ is fixed then we will use   the notation $M_{n,v}:=M_{n,v}(\l)$. If $v$ is also fixed, we will use the notation $M_{n}:=M_{n,v}$.

\begin{prop}
\label{equality for almost all non-archimedean places}
Let $v\in\Omega_\Q$ be a non-archimedean place such that
\begin{enumerate}
\item[(i)] each coefficient of $\bfA$ and of $\bfB$ are $v$-adic integers; 
\item[(ii)] the resultant of the polynomials $\bfA$ and $\bfB$, and the leading coefficients of both $\bfA$ and of $\bfB$ are $v$-adic units; and 
\item[(iii)] if the constant coefficient $a_0$ of $\bfA$ is nonzero, then $a_0$ is a $v$-adic unit. 
\end{enumerate}
Then for each $\l\in\Qbar^*$ we have $\frac{\log M_{\bfc,n,v}(\l)}{d^n} =\frac{\log M_{\bfc,1,v}(\l)}{d}$, for all $n\ge 1$. 
\end{prop}

\begin{remarks}
\label{geometric interpretation of the condition}
\begin{enumerate}
\item[(1)] Since we assumed $\bfA$ and $\bfB$ are nonzero, then conditions (i)-(iii) are satisfied by all but finitely many places $v\in\Omega_\Q$. 
\item[(2)] Conditions (i)-(ii) of Proposition~\ref{equality for almost all non-archimedean places} yield that $\bfc(t)=\bfA(t)/\bfB(t)$ has good reduction at $v$. On the other hand, if $\bfA(t)/t\bfB(t)$ has good reduction at $v$, then condition (iii) must hold. 
\end{enumerate}
\end{remarks}

\begin{proof}
Let $\l\in\Qbar^*$, let $|\cdot |:=|\cdot |_v$, let $A_n:=\bfA_{\bfc,n}(\l)$,  $B_n:=\bfB_{\bfc,n}(\l)$, and $M_n:=\max\{|A_n|, |B_n|\}$.  

Assume first that $|\l|>1$. Using conditions (i)-(ii), then $M_0=|\l|^{\deg(\bfc)}$. If $\bfc$ is nonconstant, then $M_0>1$; furthermore, for each $n\ge 1$ we have $|A_n|>|B_n|$ (because $\deg(\bfA_{\bfc,n}(t))>\deg(\bfB_{\bfc,n}(t))$ for $n\ge 1$), 
and so, $M_n= M_1^{d^{n-1}}$ for all $n\ge 1$. On the other hand, if $\bfc$ is constant, then $|\bfA_1|=|\l|>|\bfB_1|=1$, and then again for each $n\ge 1$ we have $M_n=M_1^{d^{n-1}}$. Hence Proposition~\ref{equality for almost all non-archimedean places} holds when $|\l|>1$.

Assume $|\l|\le 1$. Then it is immediate that $M_n\le 1$ for all $n\ge 0$. On the other hand, because $v$ is a place of good reduction for $\bfc$, we get that $M_0=1$. Then, using that $|\l|=1$ we obtain
$$|\bfA_1(\l)|=|\bfA(\l)^d+\l\bfB(\l)^d|\text{ and }|\bfB_1(\l)|=|\bfA(\l)\bfB(\l)^{d-1}|.$$
Then Claim~\ref{useful claim} yields that $M_n=1$ for all $n\ge 1$, and so Proposition~\ref{equality for almost all non-archimedean places} holds when $|\l|=1$.

Assume now that $|\l|<1$, then either $|\bfA(\l)|=1$ or $|\bfA(\l)|<1$. If the former holds, then first of all we note that $\bfA(0)\ne 0$ since otherwise $|\bfA(\l)|\le |\l|<1$. An easy induction yields that $|A_n|=1$ for all $n\ge 0$ (since $|B_n|\le 1$ and $|\l|<1$). Therefore, $M_n=1$ for all $n\ge 0$. Now if $|\bfA(\l)|<1$, using that $|\l|<1$, we obtain that $a_0=0$. Indeed, if $a_0$ were nonzero, then $|a_0|=1$ by our hypothesis (iii), and thus $|\bfA(\l)|=|a_0|=1$. So, indeed $\bfA(0)=0$, which yields that
\begin{equation}
\label{formula for A_1}
A_1=\frac{\bfA(\l)^d}{\l} + \bfB(\l)^d.
\end{equation}
On the other hand, since $v$ is a place of good reduction for $\bfc$, and $|\bfA(\l)|<1$ we conclude that $|\bfB(\l)|=1$. Thus  \eqref{formula for A_1} yields that $|A_1|=1$ because $d\ge 2$ and $ |\bfA(\l)|\le|\l|<1$. Because for each $n\ge 1$ we have $A_{n+1}=A_n^d+\l\cdot B_n^d$ and $|\l|<1$, while $|B_n|\le 1$, an easy induction yields that $|A_n|=1$ for all $n\ge 1$. 

This concludes the proof of Proposition~\ref{equality for almost all non-archimedean places}. 
\end{proof}

The following result is the key for our proof of Theorem~\ref{variation of canonical height}.
\begin{prop}
\label{each place}
Let $v\in\Omega_\Q$.  
There exists a positive real number $C_{v,\bfc}$ depending only on $v$, and on the coefficients of $\bfA$ and of $\bfB$ (but independent of $\l$) such that 
\begin{align*}
	& \left|\lim_{n\to\infty} \frac{\log\max\{|\bfA_{\bfc,n}(\l)|_v, |\bfB_{\bfc,n}(\l)|_v\}}{d^n} - \frac{\log\max\{|\bfA_{\bfc,2}(\l)|_v, |\bfB_{\bfc,2}(\l)|_v\}}{d^2}\right|\\
& \le C_{v,\bfc},
\end{align*}
for all $\l\in\Qbar^*$ such that $\bfc(\l)\ne 0,\infty$. 
\end{prop}

Propositions~\ref{equality for almost all non-archimedean places} and \ref{each place} yield Theorem~\ref{variation of canonical height}.

\begin{proof}[Proof of Theorem~\ref{variation of canonical height}.]
First of all we deal with the case that either $\bfA$ or $\bfB$ is the zero polynomial, i.e. $\bfc=0$ or $\bfc=\infty$ identically. In both cases, we obtain that $\bfB_{\bfc,n}=0$ for all $n\ge 1$, i.e., $\bfc$ is preperiodic for $\bff$ being always mapped to $\infty$. Then the conclusion of Theorem~\ref{variation of canonical height} holds trivially since $\hhat_{\bff_\l}(\bfc(\l))=0=\hhat_\bff(\bfc)$.

Secondly, assuming that both $\bfA$ and $\bfB$ are nonzero polynomials, we deal with the values of $\l$ excluded from the conclusion of Proposition~\ref{each place}. Since there are finitely many $\l\in\Qbar$ such that either $\l=0$ or $\bfA(\l)=0$ or $\bfB(\l)=0$ we see that the conclusion of Theorem~\ref{variation of canonical height} is not affected by these finitely many values of the parameter $\l$; the difference between $\hhat_{\bff_\l}(\bfc(\l))$ and $\hhat_\bff(\bfc)\cdot h(\l)$ can be absorbed in $O(1)$ for those finitely many values of $\l$. So, from now on we assume that $\l\in\Qbar^*$ such that $\bfc(\l)\ne 0,\infty$.

For each $\sigma\in\Gal(\Qbar/\Q)$ let $S_{\bfc^\sigma}$ be the finite set of places $v\in\Omega_\Q$ such that either $v$ is archimedean, or $v$ does not satisfy the hypothesis of Proposition~\ref{equality for almost all non-archimedean places} with respect to  $\bfc^\sigma$. Let $S=\bigcup S_{\bfc^\sigma}$, and let $C$ be the maximum of all constants $C_{v,\bfc^\sigma}$ (from Proposition~\ref{each place}) over all $v\in S$ and all $\sigma\in\Gal(\Qbar/\Q)$.  Thus from Propositions~\ref{equality for almost all non-archimedean places} and \ref{each place} we obtain for each $\l\in\Qbar^*$  such that $\bfA(\l),\bfB(\l)\ne 0$ we have
\begin{align*}
& \left|\frac{h([\bfA_{\bfc,2}(\l):\bfB_{\bfc,2}(\l)])}{d^2} - \hhat_{f_\l}(\bfc(\l)) \right| \\
& = \left|\frac{1}{[K(\l):\Q]}\sum_{\sigma}\sum_{v\in\Omega_\Q} \frac{\log\max\{|\bfA_{\bfc^\sigma, 2}(\sigma(\l))|_v, |\bfB_{\bfc^\sigma, 2}(\sigma(\l))|_v\}}{d^2} - \hhat_{\bff_{\sigma(\l)},v}(\bfc^\sigma(\sigma(\l))) \right|\\
& \le \frac{1}{[K(\l):\Q]} \sum_{\sigma} \sum_{v\in S} \left| \frac{\log\max\{|A_{\bfc^\sigma, 2}(\sigma(\l))|_v, |B_{\bfc^\sigma, 2}(\sigma(\l))|_v\}}{d^2} - \hhat_{\bff_{\sigma(\l)},v}(\bfc^\sigma(\sigma(\l))) \right|\\
& \le C\cdot |S|,
\end{align*}
where the outer sum is over all embeddings $\sigma:K(\l)\lra\Qbar$.

Finally, since the rational map $t\mapsto g_2(t):=\frac{\bfA_{\bfc,2}(t)}{\bfB_{\bfc,2}(t)}$ has degree $d^2\cdot \hhat_\bff(\bfc)$ (see Propositions~\ref{canonical height generic nonzero} and \ref{canonical height generic zero}), \cite[Theorem~1.8]{Lang-diophantine} yields that there exists a constant $C_1$ depending only on $g_2$ (and hence only on the coefficients of $\bfc$) such that for each $\l\in\Qbar$ we have: 
\begin{equation}
\label{close to the height}
\left|\frac{h([\bfA_{\bfc,2}(\l):\bfB_{\bfc,2}(\l)])}{d^2} - \hhat_\bff(\bfc)\cdot h(\l)\right|\le C_1.
\end{equation}
Using inequality \eqref{close to the height} together with the inequality
$$\left|\frac{h([\bfA_{\bfc,2}(\l):\bfB_{\bfc,2}(\l)])}{d^2} - \hhat_{\bff_\l}(\bfc(\l)) \right|\le C\cdot |S|,$$
we conclude the proof of Theorem~\ref{variation of canonical height} (note that $S$ depends only on $\bfc$). 
\end{proof}


\section{The case of constant starting point}
\label{section e=0}

In this Section we complete the proof of Proposition~\ref{each place} in the case $\bfc$ is a nonzero constant, and then proceed to proving Theorem~\ref{precise constant}. 
We start with several useful general results (not only for the case $\bfc$ is constant).

\begin{prop}
\label{fundamental inequality}
Let $m$ and $M$ be positive real numbers, let $d\ge 2$ and $k_0\ge 0$ be   integers, and let $\{N_k\}_{k\ge 0}$ be a sequence of positive real numbers. If 
$$m\le \frac{N_{k+1}}{N_k^d}\le M$$
for each $k\ge k_0$, then
$$\left|\lim_{n\to\infty}\frac{\log N_k}{d^k} - \frac{\log N_{k_0}}{d^{k_0}}\right| \le \frac{\max\{-\log(m), \log(M)\}}{d^{k_0}(d-1)}.$$   
\end{prop}

\begin{proof}
We obtain that for each $k\ge k_0$ we have
$$\left|\frac{\log N_{k+1}}{d^{k+1}} - \frac{\log N_k}{d^k} \right| \le \frac{\max\{-\log(m), \log(M)\}}{d^{k+1}}.$$
The conclusion follows by adding the above inequalities for all $k\ge k_0$. 
\end{proof}

We let $|\cdot |_v$ be an absolute value on $\Qbar$. 
As before, for each  $\bfc(t)\in\Qbar(t)$ and for each $t=\l\in\Qbar$ we let $M_{\bfc,n,v}(\l):=\max\{|\bfA_{\bfc,n}(\l)|_v, |\bfB_{\bfc,n}(\l)|_v\}$ for each $n\ge 0$. 

\begin{prop}
\label{bounded lambda}
Consider $\l\in\Qbar^*$ and $|\cdot|_v$ an absolute value on $\Qbar$. Let $m\le 1\le M$ be positive real numbers. If $m\le |\l|_v\le M$, then for each $1\le n_0\le n$ we have
$$\left|\frac{\log M_{n,v}(\l)}{d^n} - \frac{\log M_{n_0,v}(\l)}{d^{n_0}}\right| \le \frac{\log(2M)-\log(m)}{d^{n_0}(d-1)} .$$
\end{prop}

Using the classical telescoping argument, we conclude that for each $\l\in\Qbar^*$, the limit $\lim_{n\to\infty} \frac{\log M_{n,v}}{d^n}$ exists.
\begin{cor}
\label{the limit exists}
Consider $\l\in\Qbar^*$ and $|\cdot|_v$ an absolute value on $\Qbar$. Then for each $n_0\ge 1$ we have
$$\left|\lim_{n\to\infty} \frac{\log M_{n,v}(\l)}{d^n} - \frac{\log M_{n_0,v}(\l)}{d^{n_0}}\right| \le \frac{\log(2\max\{1, |\l|_v\})-\log(\min\{1, |\l|_v\})}{d^{n_0}(d-1)} .$$
\end{cor}

\begin{proof}[Proof of Proposition~\ref{bounded lambda}.]
We let $A_n:=\bfA_{\bfc,n}(\l)$, $B_n:=\bfB_{\bfc,n}(\l)$ and $M_{n,v}:=M_{n,v}(\l)$.
\begin{lemma}
\label{upper bound lemma}
Let $\l\in\Qbar^*$ and let $|\cdot|_v$ be an absolute value on $\Qbar$. If $|\l|_v\le M$, then for each $n\ge 1$, we have $M_{n+1,v}\le (M+1)\cdot M_{n,v}^d$.
\end{lemma}

\begin{proof}[Proof of Lemma~\ref{upper bound lemma}.]
Since $|\l|_v\le M$,  we have that for each $n\in\mathbb{N}$, $|A_{n+1}|_v\le (M+1)\cdot M_{n,v}^d$ and also $|B_{n+1}|_v\le M_{n,v}^d$; so
\begin{equation}
\label{upper bound recursive 0}
M_{n+1,v}\le (M+1)\cdot M_{n,v}^d,
\end{equation}
for each $n\ge 1$.
\end{proof}
Because $M\ge 1$, Lemma~\ref{upper bound lemma} yields that 
\begin{equation}
\label{uper bound recursive}
M_{n+1,v}\le 2M\cdot M_{n,v}^d.
\end{equation}
The following result will finish our proof.
\begin{lemma}
\label{lower bound in terms of lambda}
If $\l\in\Qbar^*$ and $|\cdot|_v$ is an absolute value on $\Qbar$, then for each $n\ge 1$ we have $M_{n+1,v}\ge \frac{\min\{|\l|_v,1\}}{2\max\{|\l|_v,1\}}\cdot M_{n,v}^d$. 
\end{lemma}

\begin{proof}[Proof of Lemma~\ref{lower bound in terms of lambda}.]
We let $\ell:=\min\{|\l|_v,1\}$ and $L:=\max\{|\l|_v,1\}$. Now, if 
$$\left(\frac{2L}{\ell}\right)^{\frac{1}{d}}\cdot |B_n|_v \ge |A_n|_v\ge \left(\frac{\ell}{2L}\right)^{\frac{1}{d}}\cdot |B_n|_v,$$
then $M_{n+1,v}\ge |B_{n+1}|_v\ge (\ell/2L)^{(d-1)/d}\cdot M_{n,v}^d$ (note that $\ell<2L$). On the other hand, if
$$\text{either }\left|\frac{A_n}{B_n}\right|_v> \left(\frac{2L}{\ell}\right)^{\frac{1}{d}} \text{ or } \left|\frac{A_n}{B_n}\right|_v< \left(\frac{\ell}{2L}\right)^{\frac{1}{d}}$$
then $M_{n+1,v}\ge |A_{n+1}|_v>(\ell/2L)\cdot M_{n,v}^d$. Indeed, if $|A_n/B_n|_v>(2L/\ell)^{1/d}>1$ then
$$|A_{n+1}|_v> |A_n|_v^d\cdot \left(1-|\l|_v\cdot\frac{\ell}{2L}\right)\ge M_{n,v}^d\cdot \left(1-\frac{\ell}{2}\right)\ge \frac{\ell}{2}\cdot M_{n,v}^d\ge \frac{\ell}{2L}\cdot M_{n,v}^d.$$
Similarly, if $|A_n/B_n|_v<(\ell/2L)^{1/d}<1$ then
$$|A_{n+1}|_v> |B_n|_v^d\cdot \left(|\l|_v - \frac{\ell}{2L}\right)\ge M_{n,v}^d\cdot \left(\frac{\ell}{L}-\frac{\ell}{2L}\right)=\frac{\ell}{2L}\cdot M_{n,v}^d.$$
In conclusion, we get $\frac{\ell}{2L}\cdot M_{n,v}^d \le M_{n+1,v}$ for all $n$.
\end{proof}
Lemmas \ref{upper bound lemma} and \ref{lower bound in terms of lambda}, and Proposition~\ref{fundamental inequality} finish the proof of Proposition~\ref{bounded lambda}.
\end{proof}

The next result shows that Proposition~\ref{each place} holds when $\bfc$ is a constant $\alpha$, and moreover $|\alpha|_v$ is large compared to $|\l|_v$. In addition, this result holds for $d>2$; the case $d=2$ will be handled later in Lemma~\ref{d=2 large lambda 2}.
\begin{prop}
\label{large lambda archimedean}
Assume $d\ge 3$. Let $M\ge 1$ be a real number, let $|\cdot|_v$ be an absolute value on $\Qbar$, let $\l,\alpha\in\Qbar$, let $A_n:=A_{\l,\alpha,n}$, $B_n:=B_{\l,\alpha,n}$ and $M_{n,v}:=\max\{|A_n|_v, |B_n|_v\}$ for $n\ge 0$. Let $n_0$ be a nonnegative integer. If $|\alpha|_v\ge |\l|_v/M\ge 2M$ then for $0\le n_0\le n$ we have
$$\left|\frac{\log M_{n,v}}{d^n} - \frac{\log M_{n_0,v}}{d^{n_0}}\right| \le \frac{\log(2)}{d^{n_0}(d-1)}.$$ 
\end{prop}

In particular, since we know that for each given $\l$, the limit $\lim_{n\to\infty} \frac{\log M_{n,v}}{d^n}$ exists, we conclude that
$$\left|\lim_{n\to\infty}\frac{\log M_{n,v}}{d^n} - \frac{\log M_{n_0,v}}{d^{n_0}}\right| \le \frac{\log(2)}{d^{n_0}(d-1)}.$$

\begin{proof}[Proof of Proposition~\ref{large lambda archimedean}.]
We prove by induction on $n$ the following key result. 
\begin{lemma}
\label{growth for large lambda e=0}
For each $n\ge 0$ , we have $|A_n|_v\ge \frac{|\l|_v}{M}\cdot |B_n|_v$.
\end{lemma}
\begin{proof}[Proof of Lemma~\ref{growth for large lambda e=0}.]
Set $|\cdot|:=|\cdot|_v$. The case $n=0$ is obvious since $A_0=\alpha$ and $B_0=1$.  
Now assume $|A_n|\ge \frac{|\l|}{M}\cdot |B_n|$ and we prove the statement for $n+1$. Indeed, using that $|\l|\ge 2M^2$ and $d\ge 3$ we obtain
\begin{align*}
& |A_{n+1}|-\frac{|\l|}{M}\cdot |B_{n+1}|\\
& \ge |A_n|^d - |\l|\cdot |B_n|^d - \frac{|\l|}{M}\cdot |A_n|\cdot |B_n|^{d-1}\\
& = |A_n|^d\cdot \left(1-|\l|\cdot \frac{|B_n|^d}{|A_n|^d} - |\l|\cdot \frac{|B_n|^{d-1}}{|A_n|^{d-1}}\right)\\
& \ge  |A_n|^d \cdot \left( 1- M^d\cdot |\l|^{1-d} - M^{d-1}|\l|^{2-d}\right)\\
& \ge |A_n|^d \cdot \left( 1- M^{2-d}\cdot 2^{1-d}-M^{3-d}\cdot 2^{2-d}\right)\\
& \ge |A_n|^d\cdot \left(1-2^{-2}-2^{-1}\right)\\
& \ge 0,
\end{align*} 
as desired.
\end{proof}
Lemma~\ref{growth for large lambda e=0} yields that $M_{n,v}=|A_n|_v$ for each $n$ (using that $|\l|_v/M\ge 2M>1$). Furthermore, Lemma~\ref{growth for large lambda e=0} yields 
$$\left|M_{n+1,v}-M_{n,v}^d\right| \le|\l\cdot B_n^d|_v\le M_{n,v}^d\cdot M^d|\l|_v^{1-d}\le M_{n,v}^d\cdot M^{2-d}\cdot 2^{1-d}\le \frac{1}{4}\cdot M_{n,v}^d,$$
because $|\l|_v\ge 2M^2$, $M\ge 1$ and $d-1\ge 2$.
Thus for each $n\ge 1$ we have
\begin{equation}
\label{growth of iterate large lambda e=0}
\frac{3}{4}\le \frac{M_{n+1,v}}{M_{n,v}^d}\le \frac{5}{4}.
\end{equation}
Then Proposition~\ref{fundamental inequality} yields the desired conclusion.
\end{proof}

The next result yields the conclusion of Proposition~\ref{each place} for when the starting point $\bfc$ is constant equal to $\alpha$, and $d$ is larger than $2$.

\begin{prop}
\label{e=0}
Assume $d>2$. Let $\alpha,\l\in\Qbar^*$, let $|\cdot |_v$ be an absolute value, and for each $n\ge 0$ let $A_n:=A_{\l,\alpha,n}$, $B_n:=B_{\l,\alpha, n}$ and $M_{n,v}:=\max\{|A_n|_v, |B_n|_v\}$. Consider $L:=\max\{|\alpha|_v, 1/|\alpha|_v\}$. Then for all $n_0\ge 1$ we have
$$\left|\lim_{n\to\infty} \frac{\log M_{n,v}}{d^n} - \frac{\log M_{n_0,v}}{d^{n_0}}\right| \le (3d-2)\log(2L).$$
\end{prop}

\begin{proof}
We split our proof into three cases: $|\l|_v$ is large compared to $|\alpha|_v$; $|\l|_v$ and $|\alpha|_v$ are comparable, and lastly, $|\l|_v$ is very small. We start with the case $|\l|_v\gg |\alpha|_v$. Firstly, we note  $L=\max\left\{|\alpha|_v, |\alpha|_v^{-1}\right\}\ge 1$. 

\begin{lemma}
\label{lambda large e=0}
If $|\l|_v>8L^d$ then for integers $1\le n_0\le n$ we have
\begin{equation}
\label{inequality for large lambda}
\left|\frac{\log M_{n,v}}{d^n} - \frac{\log M_{n_0,v}}{d^{n_0}}\right| \le \frac{\log(2)}{d^{n_0}(d-1)}.
\end{equation}
\end{lemma}

\begin{proof}[Proof of Lemma~\ref{lambda large e=0}.]
Since $|\l|_v>8L^d$, then $|\alpha|_v^{d-1}<\frac{|\l|_v}{2|\alpha|_v}$ and therefore
$$|f_\l(\alpha)|_v=\left|\alpha^{d-1}+\frac{\l}{\alpha}\right|_v>\frac{|\l|_v}{2|\alpha|_v} \ge \frac{|\l|_v}{2L}\ge 4L.$$
This allows us to apply Proposition~\ref{large lambda archimedean} coupled with \eqref{conversion} (with $k_0=1$) and obtain that for all $1\le n_0\le n$ we have
\begin{align*}
& \left|\frac{\log M_{n,v}}{d^n} - \frac{\log M_{n_0,v}}{d^{n_0}}\right|\\
& = \frac{1}{d}\cdot \left| \frac{\log\max\left\{|A_{\l,f_\l(\alpha),n-1}|_v, |B_{\l,f_\l(\alpha),n-1}|_v\right\}}{d^{n-1}} - \frac{\log\max\left\{|A_{\l,f_\l(\alpha), n_0-1}|_v, |B_{\l,f_\l(\alpha),n_0-1}|_v\right\}}{d^{n_0-1}}\right|\\
& \le \frac{1}{d}\cdot \frac{\log(2)}{d^{n_0-1}(d-1)},
\end{align*}
as desired.
\end{proof}

Let $R=\frac{1}{4^dL^d}$. If $R\le |\l|_v\le 8L^d$, then Proposition~\ref{bounded lambda} yields that for all $1\le n_0\le n$ we have 
\begin{equation}
\label{inequality for middle lambda}
\left|\frac{\log M_{n,v}}{d^n}-\frac{\log M_{n_0,v}}{d^{n_0}}\right| \le \frac{2d\log(4L)}{d^{n_0}(d-1)}\le \log(4L).
\end{equation}

So we are left to analyze the range $|\l|_v< R$.  

\begin{lemma}
\label{no n_1}
If $|\l|_v<R$, then $\left| \frac{\log M_{n,v}}{d^n} - \frac{\log M_{n_0,v}}{d^{n_0}}\right| \le (3d-2)\log(2L)$ for all  integers $0\le n_0\le n$.
\end{lemma}

\begin{proof}[Proof of Lemma~\ref{no n_1}.]
Firstly we note that since $|\l|_v<R<1$, Lemma~\ref{upper bound lemma} yields that $M_{n+1,v}\le 2\cdot M_{n,v}^d$ and arguing as in the proof of Proposition~\ref{fundamental inequality} we obtain that for all $0\le n_0\le n$ we have
\begin{equation}
\label{right hand side is fine}
\frac{\log M_{n,v}}{d^n} -\frac{\log M_{n_0,v}}{d^{n_0}} \le \frac{\log(2)}{d^{n_0}(d-1)}.
\end{equation}
Next, we will establish a lower bound for the main term from \eqref{right hand side is fine}.  Since 
$$|f_\l^0(\alpha)|_v=|\alpha|_v\ge \frac{1}{L}>\sqrt[d]{2R}>\sqrt[d]{2|\l|_v},$$ 
we conclude that the smallest integer $n_1$ (if it exists) satisfying $|f_\l^{n_1}(\alpha)|_v<\sqrt[d]{2|\l|_v}$ is positive. We will now derive a lower bound for $n_1$ (if $n_1$ exists) in terms of $L$.

We know that for all $n\in\{0,\dots, n_1-1\}$ we have $|f_\l^n(\alpha)|_v\ge \sqrt[d]{2|\l|_v}$. Hence, for each $0\le n\le n_1-1$ we have
\begin{equation}
\label{(1)}
|A_{n+1}|_v \ge |A_n|_v^d\cdot \left(1- \frac{|\l|_v}{|f_\l^n(\alpha)|_v^d}\right) \ge \frac{|A_n|_v^d}{2}.  
\end{equation}
On the other hand, 
\begin{equation}
\label{(2)}
\frac{|\l|_v}{|f_\l^n(\alpha)|_v} \le \frac{|f_\l^n(\alpha)|_v^{d-1}}{2}. 
\end{equation}
So, for each $0\le n\le n_1-1$ we have
\begin{equation}
\label{(3)}
|f_\l^{n+1}(\alpha)|_v \ge |f_\l^n(\alpha)|_v^{d-1} - \frac{|\l|_v}{|f_\l^n(\alpha)|_v} \ge \frac{|f_\l^n(\alpha)|_v^{d-1}}{2}.
\end{equation}
Therefore, repeated applications of \eqref{(3)} yield that for $0\le n\le n_1$ we have
\begin{equation}
\label{(4)}
|f_\l^n(\alpha)|_v \ge \frac{|\alpha|_v^{(d-1)^n}}{2^{\frac{(d-1)^n-1}{d-2}}}\ge  \frac{1}{L^{(d-1)^n}\cdot 2^{\frac{(d-1)^n-1}{d-2}}}\ge \frac{1}{(2L)^{(d-1)^n}},  
\end{equation}
because $|\l|_v\ge 1/L$ and $d-2\ge 1$. So, if $|f_\l^{n_1}(\alpha)|_v<\sqrt[d]{2|\l|_v}$, then
$$\frac{1}{(2L)^{(d-1)^{n_1}}} < \sqrt[d]{2|\l|_v}.$$
Using now the fact that $\log(2)< \log(2L)\cdot (d-1)^{n_1}$ and that $d\le (d-1)^2-1$ (since $d\ge 3$) we obtain 
\begin{equation}
\label{(9)}
\log\left(\frac{1}{|\l|_v}\right) < \log(2L)\cdot (d-1)^{n_1+2}.  
\end{equation}
Moreover, inequality \eqref{(4)} yields that for each $0\le n\le n_1-1$, we have
\begin{equation}
\label{(5)}
|B_{n+1}|_v = |B_n|_v^d \cdot |f_\l^n(\alpha)|_v \ge |B_n|_v^d \cdot \frac{1}{(2L)^{(d-1)^n}}. 
\end{equation}
Combining \eqref{(1)} and \eqref{(5)} we get $M_{n+1,v} \ge \frac{M_{n,v}^d}{(2L)^{(d-1)^n}}$, if $0\le n\le n_1-1$. So,
\begin{equation}
\label{(6)}
\frac{\log M_{n+1,v}}{d^{n+1}} \ge \frac{\log M_{n,v}}{d^n} - \log(2L)\cdot \left(\frac{d-1}{d}\right)^n.  
\end{equation}
Summing up \eqref{(6)} starting from $n=n_0$ to $N-1$ for some $N\leq n_1$, and using inequality \eqref{right hand side is fine} we obtain that for $0\le n_0\le n\le n_1$ we have
\begin{equation}
\label{inequality for small lambda}
\left| \frac{\log M_{n,v}}{d^n} - \frac{\log M_{n_0,v}}{d^{n_0}} \right|\le d\log(2L).
\end{equation}

Now, for $n\ge n_1$, we use Lemma~\ref{lower bound in terms of lambda} and obtain
\begin{equation}
\label{(12)}
M_{n+1,v} \ge \frac{\min\{|\l|_v,1\}}{2\max\{|\l|_v,1\}}\cdot M_{n,v}^d= \frac{|\l|_v}{2}\cdot M_{n,v}^d,
\end{equation}
because $|\l|_v<R<1$. 
Inequalities \eqref{right hand side is fine} and \eqref{(12)}  yield that for all $n\ge n_0\ge  n_1$, we have
\begin{equation}
\label{(13)}
\left|\frac{\log(M_{n,v})}{d^n} - \frac{\log M_{n_0,v}}{d^{n_0}}\right| \le  \log\left(\frac{2}{|\l|_v}\right)\cdot \sum_{n=n_0}^{n-1}\frac{1}{d^{n+1}}<\frac{2\log\left(\frac{1}{|\l|_v}\right)}{(d-1)\cdot d^{n_1}}.  
\end{equation}
In establishing inequality \eqref{(13)} we also used the fact that $|\l|_v<R<1/2$ and so, $\log(2/|\l|_v)<2\log(1/|\l|_v)$.  
Combining inequalities \eqref{(9)}, \eqref{inequality for small lambda} and \eqref{(13)} yields that for all $0\le n_0\le n$ we have 
\begin{align*}
& \left|\frac{\log M_{n,v}}{d^n} - \frac{\log M_{n_0,v}}{d^{n_0}}\right|\\
& < d\log(2L)+\frac{2\cdot (d-1)^{n_1+2}\log(2L)}{(d-1)\cdot d^{n_1}}\\
& < \log(2L)\cdot (d+2\cdot (d-1))\\
& \le  (3d-2)\log(2L),
\end{align*}
as desired.

If on the other hand, we had that $|f_\l^n(\alpha)|_v\ge \sqrt[d]{2|\l|_v}$ for all $n\in\mathbb{N}$, 
we get that equation \eqref{inequality for small lambda} holds for all $n\in\mathbb{N}$. Hence, in this case too, the Lemma follows.
\end{proof}
Lemmas~\ref{lambda large e=0} and \ref{no n_1} and inequality \eqref{inequality for middle lambda} finish our proof.
\end{proof}

For $d=2$ we need a separate argument for proving Proposition~\ref{each place} when $\bfc$ is constant.
\begin{prop}
\label{e=0 d=2}
Let $d=2$, $\alpha,\l\in\Qbar^*$, let $|\cdot |_v$ be an absolute value, and for each $n\ge 0$ let $A_n:=A_{\l,\alpha,n}$, $B_n:=B_{\l,\alpha, n}$ and $M_{n,v}:=\max\{|A_n|_v, |B_n|_v\}$. Let $L:=\max\{|\alpha|_v, 1/|\alpha|_v\}$. Then for all $1\le n_0\le n$ we have
$$\left|\frac{\log M_{n,v}}{2^n} - \frac{\log M_{n_0,v}}{2^{n_0}}\right| \le 4\log(2L).$$
\end{prop}

In particular, since we know (by Corollary~\ref{the limit exists}) that the limit $\lim_{n\to\infty} \frac{\log M_{n,v}}{2^n}$ exists, we conclude that
$$\left|\lim_{n\to\infty}\frac{\log M_{n,v}}{2^n} - \frac{\log M_{n_0,v}}{2^{n_0}}\right| \le 4\log(2L).$$

\begin{proof}[Proof of Proposition~\ref{e=0 d=2}.]
We employ the same strategy as for the proof of Proposition~\ref{e=0}, but there are several technical difficulties for this case. Essentially the problem lies in the fact that $\infty$ is not a superattracting (fixed) point for  $\bff_\l(z)=\frac{z^2+\l}{z}$. So the main change is dealing with the case when $|\l|_v$ is large, but there are changes also when dealing with the case $|\l|_v$ is close to $0$.

\begin{lemma}
\label{d=2 large lambda 2}
Assume $|\l|_v>Q:=(2L)^4$. Then for integers $1\le n_0\le n$ we have 
$$\left|\frac{\log M_{n,v}}{2^n} - \frac{\log M_{n_0,v}}{2^{n_0}}\right| <\frac{5}{2}.$$
\end{lemma}

\begin{proof}[Proof of Lemma~\ref{d=2 large lambda 2}.]
Let $k_1$ be the smallest positive integer (if it exists) such that $|f_\l^{k_1}(\alpha)|_v< \sqrt{2|\l|_v}$. So, we know that $|f_\l^n(\alpha)|_v\ge \sqrt{2|\l|_v}$ if $1\le n\le k_1-1$.  We will  show that $k_1>\log_4\left(\frac{|\l|_v}{4L^2}\right)\ge 1$ (note that $|\l|_v>Q=(2L)^4$). 
\begin{claim}
\label{k_1 is large}
For each positive integer $n\le \log_4\left(\frac{|\l|_v}{2L}\right)$ we have $|f_\l^n(\alpha)|_v\ge \frac{|\l|_v}{2^{n}L}$.
\end{claim}

\begin{proof}[Proof of Claim~\ref{k_1 is large}.]
The claim follows by induction on $n$; the case $n=1$ holds since $|\l|_v>(2L)^4$ and so,
$$|f_\l(\alpha)|_v\ge \frac{|\l|_v}{|\alpha|_v}-|\alpha|_v\ge \frac{|\l|_v}{L}-L\ge \frac{|\l|_v}{2L}.$$
Now, assume for $1\le n\le \log_4\left(\frac{|\l|_v}{2L}\right)$ we have $|f_\l^n(\alpha)|_v\ge \frac{|\l|_v}{2^{n}L}$. Then
$$|f_\l^{n+1}(\alpha)|_v\ge |f_\l^n(\alpha)|_v-\frac{|\l|_v}{|f_\l^n(\alpha)|_v} \ge \frac{|\l|_v}{2^nL}-2^nL\ge \frac{|\l|_v}{2^{n+1}L},$$
because $|\l|_v\ge 4^n\cdot 2L$ since $n\le \log_4\left(\frac{|\l|_v}{2L}\right)$. This concludes our proof.
\end{proof}

Claim~\ref{k_1 is large} yields that for each $1\le n\le \log_4\left(\frac{|\l|_v}{4L^2}\right)<
\log_4\left(\frac{|\l|_v}{2L}\right)$ we have 
$$|f_\l^n(\alpha)|_v\ge \frac{|\l|_v}{2^nL}\ge \frac{|\l|_v}{\sqrt{\frac{|\l|_v}{4L^2}}\cdot L}>\sqrt{2|\l|_v}.$$
Hence,

\begin{equation}
\label{k_1 is indeed large}
k_1>\log_4\left(\frac{|\l|_v}{4L^2}\right). 
\end{equation}
Now for each $1\le n\le k_1-1$ we have 
\begin{equation}
\label{inequality step n 0}
\frac{|A_n|_v}{|B_n|_v}=|f_\l^n(\alpha)|_v\ge \sqrt{2|\l|_v}>1,
\end{equation} 
because $|\l|_v>Q>2$ and so, $M_{n,v}=|A_n|_v$. Furthermore, 
$$|f_\l^{k_1}(\alpha)|_v\ge |f_\l^{k_1-1}(\alpha)|_v - \frac{|\l|_v}{|f_\l^{k_1-1}(\alpha)|_v} \ge \sqrt{2|\l|_v} - \sqrt{\frac{|\l|_v}{2}}=\sqrt{\frac{|\l|_v}{2}}>1.$$
Hence $M_{k_1}=|A_{k_1}|_v$ and therefore, for each $1\le n\le k_1-1$, using \eqref{inequality step n 0} we have
$$|M_{n+1,v}-M_{n,v}^2|\le|\l|_v\cdot |B_n|_v^2\le \frac{|A_n|_v^2}{2}=\frac{M_{n,v}^2}{2}.$$
Hence $\frac{M_{n,v}^2}{2}\le M_{n+1,v}\le \frac{3M_{n,v}^2}{2}$, and so
\begin{equation}
\label{d=2 bounds equation}
\left|\frac{\log M_{n+1,v}}{2^{n+1}}-\frac{\log M_{n,v}}{2^n}\right|<\frac{\log(2)}{2^{n+1}},
\end{equation}
for $1\le n\le k_1-1$. The next result establishes a similar inequality to \eqref{d=2 bounds equation} which is valid for all $n\in\bN$.

\begin{claim}
\label{general bounds for d=2}
For each $n\ge 1$ we have $\frac{1}{2|\l|_v}\le \frac{M_{n+1,v}}{M_{n,v}^2}\le 2|\l|_v$.
\end{claim}

\begin{proof}[Proof of Claim~\ref{general bounds for d=2}.]
The lower bound is an immediate consequence of Lemma~\ref{lower bound in terms of lambda} (note that $|\l|_v>Q>1$), while the upper bound follows from Lemma~\ref{upper bound lemma}.
\end{proof}

Using Claim~\ref{general bounds for d=2} we obtain that for all $n\ge 1$ we have
\begin{equation}
\label{general bounds equation for d=2}
\left|\frac{\log M_{n+1,v}}{2^{n+1}}-\frac{\log M_{n,v}}{2^n}\right| \le \frac{\log(2|\l|_v)}{2^{n+1}}.
\end{equation}
Using inequalities \eqref{k_1 is indeed large}, \eqref{d=2 bounds equation} and \eqref{general bounds equation for d=2} we obtain that for all $1\le n_0\le n$ we have
\begin{align}
\nonumber
& \left| \frac{\log M_{n,v}}{2^n} - \frac{\log M_{n_0,v}}{2^{n_0}}\right|\\
\nonumber
& \le \sum_{n=1}^{k_1-1} \frac{\log(2)}{2^{n+1}} + \sum_{n=k_1}^\infty \frac{\log(2|\l|_v)}{2^{n+1}}\\
\nonumber
& \le \frac{\log(2)}{2} + \frac{\log(2|\l|_v)}{2^{k_1}}\\
\nonumber
& \le \frac{\log(2)}{2} + \frac{\log(2|\l|_v)}{\sqrt{\frac{|\l|_v}{4L^2}}}\\
\nonumber
& \le \frac{1}{2} + \frac{\log(2|\l|_v)}{\sqrt[4]{|\l|_v}}\text{ (because $|\l|_v>(2L)^4$)}\\
\label{inequality for large lambda 2}
& < \frac{5}{2}\text{ (because $|\l|_v>Q\ge 16$),}
\end{align}
as desired.

If on the other hand for all $n\in\mathbb{N}$, we have that $|f_\l^n(\alpha)|_v\ge \sqrt{2|\l|_v}$, we get that equation \eqref{d=2 bounds equation} holds for all $n\in\mathbb{N}$. Hence, in this case too the Lemma follows.
\end{proof}

Let $R=\frac{1}{4L^2}$. If $R\le |\l|_v\le Q$ then for each $n_0\ge 1$, Proposition~\ref{bounded lambda} yields 
\begin{equation}
\label{inequality for middle lambda 2}
\left|\frac{\log M_{n,v}}{2^n} - \frac{\log M_{n_0,v}}{2^{n_0}}\right| \le\frac{\log(2Q)-\log(R)}{2} < \frac{7\log(2L)}{2}<4\log(2L).
\end{equation}

Next we deal with the case $|\l|_v$ is small. 
\begin{lemma}
\label{no n_1 2}
If $|\l|_v<R$, then for all $1\le n_0\le n$, we have $\left|\frac{\log M_{n,v}}{2^n} - \frac{\log M_{n_0,v}}{2^{n_0}}\right| \le 3\log(2L)$.
\end{lemma}

\begin{proof}[Proof of Lemma~\ref{no n_1 2}.]
The argument is similar to the proof of Lemma~\ref{d=2 large lambda 2}, only that this time we do not know that $|f_\l^n(\alpha)|_v>1$ (and therefore we do not know that $|A_n|_v>|B_n|_v$) because $|\l|_v$ is small. Also, the proof is similar to the proof of Lemma~\ref{no n_1}, but there are several changes due to the fact that $d=2$.

We note that since $|\l|_v<R<1$ then Lemma~\ref{upper bound lemma} yields
\begin{equation}
\label{right hand side is fine 2}
M_{n+1}\le 2M_n^2.
\end{equation}

Now, let $n_1$ be the smallest integer (if it exists) such that $|f_\l^{n_1}(\alpha)|_v<\sqrt{2|\l|_v}$. Note that $|f_\l^0(\alpha)|_v=|\alpha|_v\ge \frac{1}{L}\ge \sqrt{2|\l|_v}$ because $|\l|_v<R=\frac{1}{4L^2}$. Hence (if $n_1$ exists), we get that $n_1\ge 1$. In particular, for each $0\le n\le n_1-1$ we have $|f_\l^{n}(\alpha)|_v\ge \sqrt{2|\l|_v}$ and this yields
\begin{equation}
\label{(1) 2}
|A_{n+1}| \ge |A_n|^2\cdot \left(1- \frac{|\l|_v}{|f_\l^n(\alpha)|_v^2}\right) \ge \frac{|A_n|^2}{2}.  
\end{equation}
On the other hand,
\begin{equation}
\label{(2) 2}
\frac{|\l|_v}{|f_\l^n(\alpha)|_v} \le \frac{|f_\l^n(\alpha)|_v}{2}. 
\end{equation}
So, for each $0\le n\le n_1-1$ we have
\begin{equation}
\label{(3) 2}
|f_\l^{n+1}(\alpha)|_v \ge |f_\l^n(\alpha)|_v - \frac{|\l|_v}{|f_\l^n(\alpha)|_v} \ge \frac{|f_\l^n(\alpha)|_v}{2}.
\end{equation}
Therefore, repeated applications of \eqref{(3) 2} yield for $n\le n_1$ that
\begin{equation}
\label{(4) 2}
|f_\l^n(\alpha)|_v \ge \frac{|\alpha|_v}{2^n}\ge \frac{1}{2^nL},  
\end{equation}
because $|\alpha|_v\ge 1/L$. So, for each $n\ge 0$ we have
\begin{equation}
\label{(5) 2}
|B_{n+1}|_v = |B_n|_v^2 \cdot |f_\l^n(\alpha)|_v \ge |B_n|_v^2 \cdot \frac{1}{2^nL}. 
\end{equation}
Combining \eqref{(1) 2} and \eqref{(5) 2} we get 
\begin{equation}
\label{left hand side is fine 2}
M_{n+1} \ge \frac{M_n^2}{L\cdot 2^{\max\{1,n\}}} 
\end{equation}
for all $n\ge 0$. Using \eqref{right hand side is fine 2} and \eqref{left hand side is fine 2} we obtain for $0\le n\le n_1-1$ that
\begin{equation}
\label{no n_1 equation}
\left|\frac{\log M_{n+1}}{2^{n+1}}-\frac{\log M_n}{2^n}\right| \le \frac{\max\{1,n\}\cdot \log(2)+\log(L)}{2^{n+1}}.
\end{equation}
Summing up \eqref{no n_1 equation} starting from $n=n_0$ to $n=n_1-1$  we obtain that for $1\le n_0\le n\le n_1$ we have
\begin{equation}
\label{inequality for small lambda 22}
\left| \frac{\log M_n}{2^n} - \frac{\log M_{n_0}}{2^{n_0}} \right|\le \sum_{k=n_0}^{n-1} \frac{k\log(2)+\log(L)}{2^{k+1}}< \log(2)+\log(L)=\log(2L).
\end{equation}
Using inequality \eqref{(4) 2} for $n=n_1$ yields $\frac{1}{2^{n_1}L} \le |f_\l^{n_1}(\alpha)|_v< \sqrt{2|\l|_v}$,
and so, 
\begin{equation}
\label{(9) 2}
\frac{1}{|\l|_v} < 4^{n_1}\cdot 2L^2.  
\end{equation}
Now, for $n\ge n_1$, we use Lemma~\ref{lower bound in terms of lambda} and obtain
\begin{equation}
\label{(12) 2}
M_{n+1} \ge \frac{\min\{|\l|_v,1\}}{2\max\{|\l|_v,1\}}\cdot M_n^2= \frac{|\l|_v}{2}\cdot M_n^2,
\end{equation}
because $|\l|<R<1$. 
Inequality \eqref{(12) 2} coupled with inequality \eqref{right hand side is fine 2}  yields that for all $n\ge n_0\ge  n_1$, we have
\begin{equation}
\label{(13) 2}
\left|\frac{\log(M_n)}{2^n} - \frac{\log M_{n_0}}{2^{n_0}} \right| < \log\left(\frac{2}{|\l|_v}\right)\cdot \sum_{n=n_0}^{n-1}\frac{1}{2^{n+1}}<\frac{\log\left(\frac{2}{|\l|_v}\right)}{2^{n_1}}.  
\end{equation}
Combining inequalities \eqref{(9) 2}, \eqref{no n_1 equation} and \eqref{(13) 2} yields that for all $1\le n_0\le n$ we have 
\begin{align}
\nonumber
& \left|\frac{\log M_n}{2^n} - \frac{\log M_{n_0}}{2^{n_0}}\right|\\
\nonumber
& < \log(2L)+\frac{(n_1+1)\log(4)+2\log(L)}{2^{n_1}}\\
\nonumber
& < \log(2L)+\log(4)+\log(L)\\
\label{inequality for small lambda 222}
& \le 3\log(2L),
\end{align}
as desired.
\end{proof}
Lemmas~\ref{d=2 large lambda 2} and \ref{no n_1 2}, and inequality \eqref{inequality for middle lambda 2} finish our proof.
\end{proof}


\begin{proof}[Proof of Theorem~\ref{precise constant}.]
First we deal with the case $\alpha=0$. In this case, $\alpha=0$ is preperiodic under the action of the family $\bff_\l$ and so, $\hhat_{\bff_\l}(\alpha)=0=h(\alpha)$. From now on, assume that $\alpha\ne 0$. Secondly, if $\l=0$ (and $d\ge 3$) then $\hhat_{\bff_0}(\alpha)=h(\alpha)$ (since $\bff_0(z)=z^{d-1}$) and thus 
$$\hhat_{\bff_0}(\alpha)-\hhat_\bff(\alpha)\cdot h(\alpha)=\frac{d-1}{d}\cdot h(\alpha)\le 6d\cdot h(\alpha),$$
and so the conclusion of Theorem~\ref{precise constant} holds. So, from now on we assume both $\alpha$ and $\l$ are nonzero.

Propositions~\ref{e=0} and \ref{e=0 d=2} allow us to apply the same strategy as in the proof of Theorem~\ref{variation of canonical height} only that this time it suffices to compare $\hhat_{f_\l}(\alpha)$ and $h([A_{\l,\alpha,1}:B_{\l,\alpha,1}])$. 
As before, we let $S$ be the set of places of $\Q$ containing the archimedean place and all the non-archimedean places $v$ for which there exists some $\sigma\in\Gal(\Qbar/\Q)$ such that $|\sigma(\alpha)|_v\ne 1$. Since $\alpha\ne 0$, we have that $S$ is finite; moreover $|S|\le 1+\ell$. So, applying Proposition~\ref{equality for almost all non-archimedean places} and Propositions~\ref{e=0} and \ref{e=0 d=2} with $n_0=1$ (see also \eqref{defi local canonical height}) we obtain
\begin{align*}
& \left|\frac{h([A_{\l,\alpha,1}:B_{\l,\alpha,1}])}{d} - \hhat_{f_\l}(\alpha) \right|\\
& = \left|\frac{1}{[K(\l):\Q]}\sum_{\sigma}\sum_{v\in\Omega_\Q} \frac{\log\max\{|A_{\sigma(\l),\sigma(\alpha), 1}|_v, |B_{\sigma(\l), \sigma(\alpha), 1}|_v\}}{d} - \hhat_{\bff_{\sigma(\l)},v}(\sigma(\alpha)) \right|\\
& \le \frac{1}{[K(\l):\Q]} \sum_{\sigma} \sum_{v\in S} \left| \frac{\log\max\{|A_{\sigma(\l),\sigma(\alpha), 1}|_v, |B_{\sigma(\l), \sigma(\alpha), 1}|_v\}}{d} - \hhat_{\bff_{\sigma(\l)},v}(\sigma(\alpha)) \right|\\
& \le \frac{3d-2}{[K(\l):\Q]}\sum_{\sigma:K(\l)\lra \Qbar}\sum_{v\in S}\log\left(2\max\{|\sigma(\alpha)|_v, |\sigma(\alpha)|_v^{-1}\}\right)\\
& \le \frac{3d-2}{[K(\l):\Q]}\sum_{\sigma:K(\l)\lra \Qbar}\sum_{v\in S}\left(\log(2)+\log^+|\alpha|_v+ \log^+\left|\frac{1}{\alpha}\right|_v\right)\\
& \le (3d-2)\cdot \left(|S|+h(\alpha)+h\left(\frac{1}{\alpha}\right)\right)\\
& \le (3d-2)\cdot (1+\ell+2h(\alpha)).
\end{align*}
On the other hand, using that $\hhat_\bff(\alpha)=1/d$ (by Proposition~\ref{canonical height generic nonzero}) and also using the basic inequalities (1)-(3) for the Weil height from Subsection~\ref{heights subsection} we obtain 
\begin{align*}
& \left|\frac{h([A_{\l,\alpha,1}:B_{\l,\alpha,1}])}{d}-\hhat_\bff(\alpha)\cdot h(\l)\right|\\
& = \left|\frac{h\left(\alpha^{d-1}+\frac{\l}{\alpha}\right)}{d} - \frac{h(\l)}{d}\right|\\
& \le \frac{1}{d}\cdot \left(\left| h\left(\alpha^{d-1}+\frac{\l}{\alpha}\right) - h\left(\frac{\l}{\alpha}\right)\right| + \left| h\left(\frac{\l}{\alpha}\right) - h(\l)\right|\right)\\
& \le \frac{1}{d}\cdot ((d-1)h(\alpha)+\log(2)+h(\alpha))\\
& < h(\alpha)+1.
\end{align*}
This finishes the proof of Theorem~\ref{precise constant}.
\end{proof}

\begin{remark}
Theorem~\ref{precise constant} yields an effective method for finding all $\l\in\Qbar$ such that a given point $\alpha\in\Qbar$ is preperiodic under the action of $\bff_\l$.
\end{remark}


\section{Proof of our main result}
\label{proofs}

So we are left to proving Proposition~\ref{each place} in full generality.  We fix a place $v\in\Omega_\Q$.  
Before completing the proof of Proposition~\ref{each place} we need one more result.

\begin{prop}
\label{really small lambda}
Assume $d>2$. Let  $\alpha, \l\in\Qbar$, and let $|\cdot |_v$ be an absolute value. We let $A_n:=A_{\l,\alpha,n}$, $B_n:=B_{\l,\alpha,n}$ and $M_{n,v}:=\max\{|A_n|_v, |B_n|_v\}$.  If $|\alpha|_v\ge 2$ and  $|\l|_v\le \frac{1}{2}$ then for each $n_0\ge 0$ we have 
$$\left|\displaystyle\lim_{n\to\infty}\frac{\log M_{n,v}}{d^n} - \frac{\log M_{n_0,v}}{d^{n_0}}\right|\le \frac{\log(2)}{d^{n_0}(d-1)} .$$ 
\end{prop}

\begin{proof}
First we claim that for each $n\ge 0$ we have $|f_\l^n(\alpha)|_v\ge 2$. Indeed, for $n=0$ we have $|\alpha|_v\ge 2$ as desired. We assume $|f_\l^n(\alpha)|_v\ge 2$ and since $|\l|_v\le 1/2$ we get that
$$ |f_\l^{n+1}(\alpha)|_v
 \ge |f_\l^n(\alpha)|_v^{d-1} - \frac{|\l|_v}{|f_\l^n(\alpha)|_v}
 \ge 4-\frac{1}{4} > 2.
$$
Hence $M_{n,v}=|A_n|_v$ and we obtain that
$$\left| M_{n+1,v}-M_{n,v}^d\right| \le |\l|_v\cdot |B_n|_v^d \le \frac{M_{n,v}^d}{2\cdot \left|\frac{A_n}{B_n}\right|_v^d}= \frac{M_{n,v}^d}{2|f_\l^n(\alpha)|_v^d} \le \frac{M_{n,v}^d}{16}$$
because $d\ge 3$. 
Thus Proposition~\ref{fundamental inequality} yields the desired conclusion.
\end{proof}

The next result deals with the case $d=2$ in Proposition~\ref{really small lambda}.
\begin{prop}
\label{d=2 really small lambda}
Assume $d=2$. Let $M\ge 1$ be a real number, let $\alpha, \l\in\Qbar$, and let $|\cdot |_v$ be an absolute value. We let $A_n:=A_{\l,\alpha,n}$, $B_n:=B_{\l,\alpha,n}$ and $M_{n,v}:=\max\{|A_n|_v, |B_n|_v\}$.  If $|\alpha|_v\ge \frac{1}{M\cdot |\l|_v}\ge 2M$ then for each $0\le n_0\le n$ we have 
$$\left|\frac{\log M_{n,v}}{2^n} - \frac{\log M_{n_0,v}}{2^{n_0}}\right|\le 4\log(2) .$$ 
\end{prop}
In particular, using Corollary~\ref{the limit exists} we obtain that for all $n_0\ge 0$ we have
$$\left|\lim_{n\to\infty}\frac{\log M_{n,v}}{2^n} - \frac{\log M_{n_0,v}}{2^{n_0}}\right|\le 4\log(2) .$$ 

\begin{proof}[Proof of Proposition~\ref{d=2 really small lambda}.]
Since $|\l|_v\le \frac{1}{2M^2}\le \frac{1}{2}$, using Lemmas~\ref{upper bound lemma} and \ref{lower bound in terms of lambda} we obtain for all $n\ge 0$ that
\begin{equation}
\label{general bound d=2 again}
\left|\frac{\log M_{n+1,v}}{2^{n+1}} - \frac{\log M_{n,v}}{2^{n}}\right|\le \frac{\log\left(\frac{2}{|\l|_v}\right)}{2^{n+1}}.
\end{equation}
We need to improve the above bound and in order to do this we prove a sharper inequality when $n$ is small compared to $\frac{1}{|\l|_v}$. The strategy is similar to the one employed in the proof of Lemma~\ref{d=2 large lambda 2}. 

First of all, since $|\l|_v<1$, Lemma~\ref{upper bound lemma} yields that for all $ n\ge 0$ we have
\begin{equation}
\label{easy inequality}
\frac{\log M_{n+1,v}}{2^{n+1}}-\frac{\log M_{n,v}}{2^{n}}\le \frac{\log(2)}{2^{n+1}}.
\end{equation}
We will prove a lower bound for the main term from \eqref{easy inequality} when $n_0$ and $n$ are small compared to $\frac{1}{|\l|_v}$. First we prove that $|f_\l^n(\alpha)|_v$ is large when $n$ is small.
\begin{lemma}
\label{the iterates are large in the beginning}
For each integer $n\le \frac{3M^2}{4|\l|_v}$, we have $|f_\l^n(\alpha)|_v\ge \frac{3M}{2}$.
\end{lemma}

\begin{proof}[Proof of Lemma~\ref{the iterates are large in the beginning}.]
We will prove the statement inductively. For $n=0$, we know that $|f_\l^0(\alpha)|_v=|\alpha|_v\ge 2M$.  If now for some $n\ge0$ we have that $|f_\l^n(\alpha)|_v\ge \frac{3M}{2}$, then $|f_\l^{n+1}(\alpha)|_v\ge |f_\l^n(\alpha)|_v-\frac{2|\l|_v}{3M}$. Therefore, for all $n\le \frac{3M^2}{4|\l|_v}$ we have
$$|f_\l^n(\alpha)|_v\ge |\alpha|_v- \frac{n\cdot 2|\l|_v}{3M^2}\ge \frac{3M}{2},$$
as desired.
\end{proof}
In conclusion, if we let $n_1$ be the smallest positive integer larger than $\frac{3M^2}{4|\l|}$ we know that for all $0\le n\le n_1-1$ we have $|f_\l^n(\alpha)|_v\ge \frac{3}{2}$. In particular, 
$$|f_\l^{n_1}(\alpha)|_v\ge |f_\l^{n_1-1}(\alpha)|_v - \frac{|\l|_v}{|f_\l^{n_1-1}(\alpha)|_v}\ge \frac{3}{2}-\frac{\frac{1}{2}}{\frac{3}{2}}>1.$$
Therefore, $M_n=|A_n|$ for all $0\le n\le n_1$, and moreover for $0\le n\le n_1-1$ we have 
\begin{equation}
\label{bound for small n M_n}
M_{n+1}-M_n^2= |A_{n+1}|_v-|A_n^2|_v\ge -|\l|_v\cdot |B_n|_v^2=-M_n^2\cdot \frac{|\l|_v}{|f_\l^n(\alpha)|_v^2}\ge -\frac{4M_n^2}{9}, 
\end{equation}
because $|\l|_v<1$ and $|f_\l^n(\alpha)|_v\ge \frac{3}{2}$. 
Inequality \eqref{bound for small n M_n} coupled with the argument from Proposition~\ref{fundamental inequality} yields that for all $0\le n\le n_1-1$ we have
\begin{equation}
\label{upper bound for small n M_n}
\frac{\log M_{n+1}}{2^{n+1}} - \frac{\log M_{n}}{2^{n}}> -\frac{\log(2)}{2^{n+1}}.
\end{equation}
Using the definition of $n_1$ and inequalities \eqref{general bound d=2 again}, \eqref{easy inequality} and \eqref{upper bound for small n M_n}, we conclude that 
\begin{align*}
& \left|\frac{\log M_n}{2^n} - \frac{\log M_{n_0}}{2^{n_0}}\right|\\
& \le \sum_{n=0}^{n_1-1}\frac{\log(2)}{2^{n+1}}+\sum_{n=n_1}^\infty \frac{\log\left(\frac{2}{|\l|}\right)}{2^{n+1}}\\
& \le \log(2)+ \frac{ \log\left(\frac{8n_1}{3M^2}\right)}{2^{n_1}}\\
& \le 4\log(2),
\end{align*}
for all $0\le n_0\le n$.
\end{proof}

Finally, we will establish the equivalent of Proposition~\ref{large lambda archimedean} for $d=2$.
\begin{prop}
\label{large lambda non-archimedean}
Assume $d=2$. Let $M\ge 1$ be a real number, let $\alpha, \l\in\Qbar$, and let $|\cdot |_v$ be an absolute value. We let $A_n:=A_{\l,\alpha,n}$, $B_n:=B_{\l,\alpha,n}$ and $M_{n,v}:=\max\{|A_n|_v, |B_n|_v\}$.  If $|\alpha|_v\ge \frac{|\l|_v}{M}\ge 8M$, then for each $0\le n_0\le n$ we have 
$$\left|\frac{\log M_{n,v}}{2^n} - \frac{\log M_{n_0,v}}{2^{n_0}}\right|\le 1+8M .$$ 
\end{prop}

In particular, using Corollary~\ref{the limit exists} we obtain that for all $n_0\ge 0$ we have
 $$\left|\lim_{n\to\infty}\frac{\log M_{n,v}}{2^n} - \frac{\log M_{n_0,v}}{2^{n_0}}\right|\le 1+8M .$$ 

\begin{proof}[Proof of Proposition~\ref{large lambda non-archimedean}.]
We know that $|\l|_v\ge 8M^2>1$. Thus, Lemmas~\ref{upper bound lemma} and \ref{lower bound in terms of lambda} yield that for all $n\ge 0$ we have
\begin{equation}
\label{non-archimedean 1}
\left|\frac{\log M_{n+1,v}}{2^{n+1}}-\frac{\log M_{n,v}}{2^n}\right|\le \frac{\log(2|\l|_v)}{2^{n+1}}.
\end{equation}
As in the proof of Proposition~\ref{d=2 really small lambda}, we will find a sharper inequality for small $n$. Arguing identically as in Claim~\ref{k_1 is large}, we obtain that for $0\le n\le \log_4\left(\frac{|\l|_v}{2M}\right)$ we have 
\begin{equation}
\label{the first iterates are large non-archimedean}
|f_\l^n(\alpha)|_v\ge \frac{|\l|_v}{2^nM}\ge 2^{n+1}\ge 2.
\end{equation}
So, let $n_1$ be the smallest integer larger than $\log_4\left(\frac{|\l|_v}{2M^2}\right)-1$. Since $|\l|_v\ge 8M^2$, we get that $n_1\ge 1$. Also, by its definition, $n_1\le \log_4\left(\frac{|\l|_v}{2M}\right)$; so, for each $0\le n\le n_1$, inequality \eqref{the first iterates are large non-archimedean} holds, and thus $M_{n,v}=|A_n|_v$. Moreover, for $0\le n\le n_1-1$ we get that
$$|M_{n+1,v}-M_{n,v}^2|=||A_{n+1}|_v-|A_n^2|_v|\le |\l|_v\cdot |B_n|_v^2= M_{n,v}^2\cdot \frac{|\l|_v}{|f_\l^n(\alpha)|_v^2}\le M_{n,v}^2\cdot \frac{|\l|_v}{\frac{|\l|_v^2}{4^nM^2}} \le \frac{M_{n,v}^2}{2}.$$
So, using Proposition~\ref{fundamental inequality} we obtain that for all $0\le n\le n_1-1$ we have
\begin{equation}
\label{inequality for small n non-archimedean}
\left|\frac{\log M_{n+1,v}}{2^{n+1}}-\frac{\log M_{n,v}}{2^n}\right| \le \frac{\log(2)}{2^{n+1}}.
\end{equation} 
Using the definition of $n_1$ and inequalities \eqref{non-archimedean 1} and \eqref{inequality for small n non-archimedean} we conclude that
\begin{align*}
& \left|\frac{\log M_{n,v}}{2^n} - \frac{\log M_{n_0,v}}{2^{n_0}}\right|\\
& \le \sum_{n=0}^{n_1-1}\frac{\log(2)}{2^{n+1}} + \sum_{n=n_1}^\infty \frac{\log(2|\l|_v)}{2^{n+1}}\\
& \le \log(2)+\frac{\log(2|\l|_v)}{2^{n_1}}\\
& \le \log(2)+ \frac{2\log(2|\l|_v)}{\sqrt{\frac{|\l|_v}{2M^2}}}\\
& \le \log(2)+ 4M\cdot \frac{\log(2|\l|_v)}{\sqrt{|\l|_v}}\\
& < 1+8M\text{ (because $|\l|_v\ge 8$),}
\end{align*}
for all $0\le n_0\le n$. 
\end{proof}

Our next result completes the proof of Proposition~\ref{each place} by considering the case of nonconstant $\bfc(t)=\frac{\bfA(t)}{\bfB(t)}$, where $\bfA,\bfB\in\Qbar[t]$ are nonzero coprime polynomials.

\begin{prop}
\label{place in S}
Assume $\bfc(\l)=\frac{\bfA(\l)}{\bfB(\l)}\in\Qbar(\l)$ is nonconstant, and let $|\cdot |_v$ be any absolute value on $\Qbar$. Consider $\l_0\in\Qbar^*$ such that $\bfc(\l_0)\ne 0,\infty$. For each $n\ge 0$ we let $\bfA_n:=\bfA_{\bfc, n}(\l_0)$, $\bfB_n:=\bfB_{\bfc,n}(\l_0)$ and $M_{n,v}:=\max\{|\bfA_n|_v, |\bfB_n|_v\}$. 
Then there exists a constant $C$ depending only on $v$ and on the coefficients of $\bfA$ and of $\bfB$ (but independent of $\l_0$) such that
\begin{equation}
\label{inequality for nonconstant}
\left|\lim_{n\to\infty} \frac{\log M_{n,v}}{d^n} - \frac{\log M_{2,v}}{d^2}\right| \le C.
\end{equation}
\end{prop}

\begin{proof}
We let $\alpha:=f_{\l_0}(\bfc(\l_0))$. Since $\l_0$ is fixed in our proof, so is $\alpha$. On the other hand, we will prove that the constant $C$ appearing in \eqref{inequality for nonconstant} does not depend on $\alpha$ (nor on $\l_0$).

We split our proof in three cases depending on $|\l_0|_v$. We first deal with the case of large $|\l_0|_v$. As proven in Propositions~\ref{canonical height generic nonzero} and \ref{canonical height generic zero}, $\deg_t(\bfA_{\bfc, 1}(t))-\deg_t(\bfB_{\bfc, 1}(t))\ge 1$. We let $c_1$ and $c_2$ be the leading coefficients of $\bfA_{\bfc,1}(t)$ and $\bfB_{\bfc,1}(t)$ respectively. Then,  there exists a positive real number $Q$ depending on $v$ and the coefficients of $\bfA$ and $\bfB$ only, such that if $|\l|_v>Q$ then 
$$\frac{|\bfA_{\bfc, 1}(\l)|_v}{|\bfB_{\bfc,1}(\l)|_v}\ge \frac{|\l|_v\cdot |c_1|_v}{2|c_2|_v}\ge 8M,$$
where $M:=2\max\{1,|c_2/c_1|_v\}$. Our first step is to prove the following result.

\begin{lemma}
\label{really large lambda}
If $|\l_0|_v>Q$ then 
\begin{equation}
\label{inequality for large lambda final}
\left|\lim_{n\to\infty}\frac{\log M_{n,v}}{d^n} - \frac{\log M_{2,v}}{d^2}\right|\le \frac{1+16\max\left\{1,\left|\frac{c_2}{c_1}\right|\right\}}{d}.
\end{equation}
\end{lemma}

\begin{proof}[Proof of Lemma~\ref{really large lambda}.]
We recall that $M:=2\max\{1,|c_2/c_1|_v\}$. 
Since $|\l_0|_v>Q$, then 
$$|\alpha|_v=\frac{|\bfA_{\bfc,1}(\l_0)|_v}{|\bfB_{\bfc,1}(\l_0)|_v}\ge |\l_0|_v/M\ge 8M.$$ 
We  apply the conclusion of Propositions~\ref{large lambda archimedean} and \ref{large lambda non-archimedean} with $n_0=1$ and we conclude that
$$\left|\lim_{n\to\infty} \frac{\log\max\{|A_{\l_0,\alpha, n}|_v, |B_{\l_0,\alpha, n}|_v\}}{d^n} - \frac{\log\max\{|A_{\l_0,\alpha, 1}|_v, |B_{\l_0,\alpha, 1}|_v\}}{d} \right| \le 1+8M.$$
On the other hand, using  \eqref{conversion} with $k_0=1$ (note that by our assumption, $\bfB_{\bfc,1}(\l_0)=\bfA(\l_0)\bfB(\l_0)^{d-1}\ne 0 $) we obtain 
\begin{align*}
& \left|\lim_{n\to\infty}\frac{\log M_{n,v}}{d^n} - \frac{\log M_{2,v}}{d^2}\right|\\
& =\frac{1}{d}\cdot \left| \lim_{n\to\infty} \frac{\log\max\{|A_{\l_0,\alpha, n}|_v, |B_{\l_0,\alpha, n}|_v\}}{d^n} - \frac{\log\max\{|A_{\l_0,\alpha, 1}|_v, |B_{\l_0,\alpha, 1}|_v\}}{d} \right|\\
& \le \frac{1+8M}{d},
\end{align*}
as desired.
\end{proof}

We will now deal with the case when $|\l_0|_v$ is small. We will define another quantity, $R$, which will depend only on $v$ and on the coefficients of $A$ and of $B$, and we will assume that $|\l_0|_v<R$. The definition of $R$ is technical since it depends on  whether $\bfc(0)$ equals $0$, $\infty$ or neither. However, the quantity $R$ will depend on $v$ and on the coefficients of $\bfA$ and of $\bfB$ only (and will not depend on $\l_0$ nor on $\alpha=f_{\l_0}(\bfc(\l_0))$).

Assume $\bfc(0)\ne 0,\infty$ (i.e., $\bfA(0)\ne 0$ and $\bfB(0)\ne 0$). Let $c_3:=\bfA(0)\ne 0$ and $c_4:=\bfB(0)\ne 0$ be  the constant coefficients of $\bfA$ and respectively of $\bfB$.  
Then there exists a positive real number $R$ depending on  $v$ and on the coefficients of $\bfA$ and of $\bfB$ only, such that if $|\l|_v<R$, then 
$$\frac{|c_3|_v}{2} < |\bfA(\l)|_v<\frac{3|c_3|_v}{2}\text{ and }\frac{|c_4|_v}{2}<|\bfB(\l)|_v<\frac{3|c_4|_v}{2}.$$
Hence $\frac{|c_3|_v}{3|c_4|_v}<|\bfc(\l)|_v<\frac{3|c_3|_v}{|c_4|_v}$.  
Then we can apply Propositions~\ref{e=0} and \ref{e=0 d=2} with $n_0=2$ (coupled with \eqref{conversion} for $k_0=0$); we obtain that if $|\l_0|_v<R$ then
\begin{align}
\nonumber
& \left|\lim_{n\to\infty}\frac{\log M_{n,v}}{d^n} - \frac{\log M_{2,v}}{d^2} \right| \\
\nonumber
& = \left|\lim_{n\to\infty}\frac{\log\max\{|A_{\l_0,\bfc(\l_0),n}|_v, |B_{\l_0,\bfc(\l_0),n}|_v\}}{d^n} - \frac{\log\max\{|A_{\l_0,\bfc(\l_0),2}|_v, |B_{\l_0,\bfc(\l_0),2}|_v\}}{d^2}
\right| \\
\nonumber
& \le (3d-2)\log\left(2\max\{|\bfc(\l_0)|_v, |\bfc(\l_0)|_v^{-1}\}\right)\\
\label{conclusion 1}
& \le (3d-2)\left(2+\log\left(\max\left\{\frac{|c_3|_v}{|c_4|_v}, \frac{|c_4|_v}{|c_3|_v}\right\}\right)\right).
\end{align}

Assume $c(0)=\infty$. Then $\bfA(0)\ne 0$ but $\bfB(0)=0$; in particular $\deg(\bfB)\ge 1$ since $\bfB$ is not identically equal to $0$. We recall that $c_3=A(0)$ is the constant coefficient of $\bfA$ (we know $c_3\ne 0$). Also, let $c_5$ be the first nonzero coefficient of $\bfB$. Then there exists a positive real number $R$ depending on $v$ and on the coefficients of $\bfA$ and of $\bfB$ only, such that if $0<|\l|_v<R$ then
$$\frac{|c_3|_v}{2}<|\bfA(\l)|_v\text{ and }|\bfB(\l)|_v<2|c_5|_v\cdot |\l|_v,$$
and moreover
$$|\bfc(\l)|_v=\left|\frac{\bfA(\l)}{\bfB(\l)}\right|_v>\frac{1}{M\cdot |\l|_v}\ge 2M,$$
where $M=4\max\{1,|c_5/c_3|_v\}$. Then applying Propositions~\ref{really small lambda} and \ref{d=2 really small lambda} with $n_0=2$ (coupled with \eqref{conversion} for $k_0=0$) we conclude that if $|\l_0|_v<R$ then
\begin{align}
\nonumber
& \left|\lim_{n\to\infty}\frac{\log M_{n,v}}{d^{n}}-\frac{\log M_{2,v}}{d^2}\right|\\
\nonumber
& =  \left| \lim_{n\to\infty}\frac{\log\max\{|A_{\l_0,\bfc(\l_0), n}|_v, |B_{\l_0,\bfc(\l_0), n}|_v\}}{d^n} -\frac{\log\max\{|A_{\l_0,\bfc(\l_0), 2}|_v, |B_{\l_0,\bfc(\l_0),2}|_v\}}{d^2}\right|\\
\label{conclusion 3'}
& \le 4\log(2). 
\end{align}

Assume $\bfc(0)=0$. Then $\bfA(0)=0$ but $\bfB(0)\ne 0$; in particular $\deg(\bfA)\ge 1$ since $\bfA$ is not identically equal to $0$. So, the constant coefficient of $\bfB$ is nonzero, i.e., $c_4=c_5=\bfB(0)\ne 0$ in this case. There are two cases: $\bfA '(0)=0$ or not. First, assume $c_6:=\bfA '(0)\ne 0$. Then there exists a positive real number $R$ depending on $v$ and on the coefficients of $\bfA$ and of $\bfB$ only such that if $0<|\l|_v<R$ then
$$|\bfA_{\bfc,1}(\l)|_v=\left|\frac{\bfA(\l)^d}{\l}+\bfB(\l)^d\right|_v\in \left(\frac{|c_4|_v^d}{2}, \frac{3|c_4|_v^d}{2}\right)\text{ and }$$
$$|\bfB_{\bfc,1}(\l)|_v=\left|\frac{\bfA(\l)\bfB(\l)^{d-1}}{\l}\right|_v\in \left(\frac{\left|c_6c_4^{d-1}\right|_v}{2},\frac{3\left|c_6c_4^{d-1}\right|_v}{2}\right).$$
Hence $\frac{|c_4|_v}{3|c_6|_v}\le |\alpha|_v\le \frac{3|c_4|_v}{|c_6|_v}$ (also note that we are using the fact that $\l_0\ne 0$ and so the above inequalities apply to our case). Hence using Propositions~\ref{e=0} and \ref{e=0 d=2} with $n_0=1$ (combined also with \eqref{conversion} for $k_0=1$, which can be used since $\bfB_{\bfc,1}(\l_0)=\bfA(\l_0)\bfB(\l_0)^{d-1}\ne 0$) we obtain 
\begin{align}
\nonumber
& \left|\lim_{n\to\infty} \frac{\log M_{n,v}}{d^n} - \frac{\log M_{2,v}}{d^2}\right|\\
\nonumber
& = \frac{1}{d}\cdot \left|\lim_{n\to\infty}\frac{\log\max\{|A_{\l_0,\alpha,n}|_v, |B_{\l_0,\alpha,n}|_v\}}{d^n} - \frac{\log\max\{|A_{\l_0,\alpha, 1}|_v, |B_{\l_0,\alpha, 1}|_v\}}{d}\right|\\
\nonumber
& \le \frac{3d-2}{d}\cdot\log\left(2\max\left\{|\alpha|_v, |\alpha|_v^{-1}\right\}\right)\\
\label{conclusion 2}
& \le 3\cdot\left(2+\log\left(\max\left\{\left|\frac{c_4}{c_6}\right|_v, \left|\frac{c_6}{c_4}\right|_v\right\}\right) \right).
\end{align}

Next assume $\bfA(0)=\bfA '(0)=0$. So, let $c_7$ be the first nonzero coefficient of $\bfA$. Also, we recall that $c_4=c_5=\bfB(0)\ne 0$ in this case. 
 Then there exists a positive real number $R$ depending on $v$ and on the coefficients of $\bfA$ and of $\bfB$ only such that if $0<|\l|_v<R$ then 
$$\frac{|c_4|_v^d}{2}< |\bfA_{\bfc,1}(\l)|_v \text{ and }|\bfB_{\bfc,1}(\l)|_v<  2\left|c_7c_4^{d-1}\right|_v\cdot |\l|_v,$$ 
and moreover
$$\left|\frac{\bfA_{\bfc,1}(\l)}{\bfB_{\bfc,1}(\l)}\right|_v>\frac{1}{M\cdot |\l|_v}\ge 8M,$$
where $M:=4\max\left\{1,\frac{|c_7|_v}{|c_4|_v}\right\}$. Hence, if $|\l_0|_v<R$ (using  also that $\bfc(\l_0)\ne 0,\infty$), we obtain
\begin{equation}
\label{last equation for alpha}
|\alpha|_v\ge \frac{1}{M\cdot |\l_0|_v}\ge 8M.
\end{equation}
Then Propositions~\ref{really small lambda} and \ref{d=2 really small lambda} with $n_0= 1$ (combined with \eqref{conversion} for $k_0=1$, which can be used since $\bfB_{\bfc,1}(\l_0)\ne 0$) yield
\begin{align}
\nonumber
& \left|\lim_{n\to\infty}\frac{\log M_{n,v}}{d^{n}}-\frac{\log M_{2,v}}{d^2}\right|\\
\nonumber
& =\frac{1}{d}\cdot  \left| \lim_{n\to\infty}\frac{\log\max\{|A_{\l_0,\alpha, n}|_v, |B_{\l_0,\alpha, n}|_v\}}{d^n} -\frac{\log\max\{|A_{\l_0,\alpha, 1}|_v, |B_{\l_0,\alpha,1}|_v\}}{d}\right|\\
& \le \frac{1+32\max\left\{1,\frac{|c_7|_v}{|c_4|_v}\right\}}{d} . 
\label{conclusion 3}
\end{align}
On the other hand, Proposition~\ref{bounded lambda} yields that if $R\le |\l_0|_v\le Q$ then
\begin{equation}
\label{inequality for middle lambda final}
\left|\lim_{n\to\infty}\frac{\log M_{n,v}}{d^n} -\frac{\log M_{2,v}}{d^2}\right| \le \frac{\log(2Q)-\log(R)}{18}.
\end{equation}
Noting that $R$ and $Q$ depend on $v$ and on the coefficients of $\bfA$ and of $\bfB$ only, inequalities \eqref{inequality for large lambda final},  \eqref{conclusion 1}, \eqref{conclusion 3'},  \eqref{conclusion 2}, \eqref{conclusion 3} and \eqref{inequality for middle lambda final} yield the conclusion of Proposition~\ref{each place}.
\end{proof}


\end{document}